\documentclass[12pt,a4paper, oneside, bold,secthm,seceqn,amsthm,ussrhead,reqno]{article}
\usepackage[utf8]{inputenc}
\usepackage[english]{babel}
\usepackage[symbol]{footmisc}
\usepackage{amssymb,amsmath,amsthm,amsfonts,xcolor,enumerate,hyperref, comment,longtable, cleveref}
\usepackage{verbatim}
\usepackage{times}
\usepackage{cite}
\usepackage{pdflscape}
\usepackage{ulem}
\usepackage[mathcal]{euscript}
\usepackage{tikz}
\usepackage{hyperref}
 
\usepackage{cancel}
\usepackage{stmaryrd}
 
\usepackage{amsfonts}
\usepackage{amssymb}
\usepackage{times}
\usepackage{xcolor}
\usepackage{tkz-graph}
\usepackage{url}
\usepackage{float}
\usepackage{tasks}
\usepackage{array}

\usepackage{cite}
\usepackage{hyperref}

 \usepackage{fancyhdr} 
\usepackage{amsfonts}
\usepackage{amsmath}
\usepackage{eurosym}
\usepackage{geometry}

\newcommand{\add}[1]{{\color{red}#1}}

\usepackage{caption,booktabs}

\theoremstyle{plain}
\newtheorem{theorem}{Theorem}
\newtheorem{lemma}[theorem]{Lemma}
\newtheorem{proposition}[theorem]{Proposition}
\newtheorem{remark}[theorem]{Remark}

\newtheorem{definition}[theorem]{Definition}

\usetikzlibrary{arrows}

\newtheorem*{theoremA}{Theorem A}
\newtheorem*{theoremB}{Theorem B}
\usepackage{fouriernc}

\begin{document}
 
 \bigskip

\noindent{\Large
The geometric classification of \\
Novikov--Poisson   and 
Gelfand--Dorfman algebras}\footnote{
The work is supported by
FCT 2023.08031.CEECIND and UID/00212/2025.}\footnote{The original concept, theoretical framework, and primary results were developed by Z.H., I.K., A.S. and Z.Z.; 
while Aristotle (https://aristotle.harmonic.fun/)   acted as a non-human research assistant, assisting in error verification and proof simplification.}

 \bigskip

\begin{center}

 {\bf
   Zeyu Hao\footnote{School of Mathematics and Statistics, Northeast Normal University, Changchun 130024, Jilin, China; \  haozeyu@nenu.edu.cn}, \ 
   Ivan Kaygorodov\footnote{CMA-UBI, University of  Beira Interior, Covilh\~{a}, Portugal; \    kaygorodov.ivan@gmail.com}, 
  Azamat Saydaliyev\footnote{Institute of Mathematics, Academy of
Sciences of Uzbekistan, Tashkent, Uzbekistan;  \ azamatsaydaliyev6@gmail.com}, \&
   Ziyi Zhang\footnote{School of Mathematics and Statistics, Northeast Normal University, Changchun 130024, Jilin, China; \ zhangziyi@nenu.edu.cn}
}

\end{center}

\ 

\noindent {\bf Abstract:}
{\it  
We give the geometric classification of complex $3$-dimensional 
Novikov--Poisson and Gelfand--Dorfman algebras. 
In particular, we prove that 
the variety of complex $3$-dimensional Novikov--Poisson algebras is $12$-dimensional, it has 4 irreducible components and has no rigid algebras; the variety of complex $3$-dimensional Gelfand--Dorfman algebras is $11$-dimensional,  it has $10$ irreducible components and has $3$ rigid algebras. 
As an additional result, we have the geometric classification of complex $3$-dimensional Novikov algebras: 
it is $9$-dimensional, and it has $4$ irreducible components with one rigid algebra.}

 \bigskip 


\bigskip 

 \
 

	 \bigskip

\ 

\


\tableofcontents 
\newpage
\section*{\large Introduction}
In recent years, Novikov--Poisson and Gelfand--Dorfman algebras have been the subject of active research 
$\big($see, for example, \cite{BBGW23, BCZ18,GS24 , KN23,KP20, KSO19, LW25,PZZ25, X96, X97, Z13, Z15,Z19,Z23,ZCM04, ZHSC-GD, ZHSC,S24 } and references therein$\big)$. 
In the present paper, we provide the geometric classification of complex $3$-dimensional Novikov--Poisson and Gelfand--Dorfman algebras. 
The algebraic classification (up to isomorphism) of algebras of dimension $n$ from a given variety defined by a specific family of polynomial identities is a classic problem in the theory of non-associative algebras \cite{FFKMN,k23,Z15,ZHSC,ZHSC-GD}. There are many results related to the algebraic classification of low-dimensional algebras in the varieties of Jordan, Lie, Leibniz, Zinbiel, and many other algebras \cite{k23}. Deformations and geometric properties of a variety of algebras defined by a family of polynomial identities have been a subject of study since the 1970s; see \cite{ADS26, AKT25, ben, BB14,FKS, FKS25, GRH} and references in \cite{k23,ls,MS}.\medskip 

The varieties of Novikov--Poisson algebras (NP-algebras) and Gelfand--Dorfman algebras (GD-algebras) are among the most well-known varieties of Poisson-type algebras. Gelfand--Dorfman algebras are closely related to differential Poisson algebras. 
In fact, if $(V,\cdot,\{-,-\})$ is a Poisson algebra with a Poisson derivation $d$, then $(V,\circ,\{-,-\})$ is a GD-algebra, where $x\circ y=xd(y)$. A GD-algebra is called special if it can be embedded into an appropriate differential Poisson algebra. The problem of describing special GD-algebras is one of the most important questions in the modern theory of non-associative algebras. For instance, Kolesnikov and Panasenko proved that for any Novikov algebra $(V,\circ)$, the algebra $(V,\circ,[x,y])$, where $[x,y]= x \circ y - y \circ x$, is a special GD-algebra \cite{KP20}. Kolesnikov and Nesterenko proved that every commutator GD-algebra obtained from a Novikov--Poisson algebra is special, i.e., embeddable into a differential Poisson algebra \cite{KN23}. Kolesnikov, Kozlov, and Panasenko proved that the commutator GD-superalgebra of a Novikov superalgebra is special, and subsequently, they proved that every quadratic Lie conformal superalgebra constructed on a finite-dimensional special GD-superalgebra has a finite faithful conformal representation \cite{KKP21}. Gubarev and Sartayev constructed the monomial basis of the free special Gelfand--Dorfman algebra \cite{GS24}. Gelfand--Dorfman algebras appeared in a classical paper by Gelfand and Dorfman in 1979 \cite{GD79}, but significant interest from a purely algebraic point of view was inspired by a paper written by Kolesnikov, Sartayev, and Orazgaliev in 2019 \cite{KSO19}. Since then, they have been under active investigation. For example, Sartayev proved that the variety of commutative GD-algebras coincides with the variety of transposed Poisson algebras \cite{S24}. Kolesnikov and Sartayev proved that the operad of GD-algebras does not satisfy the Dong property for binary quadratic operads \cite{KS26}. Xu established that a quadratic conformal superalgebra is equivalent to a GD-superalgebra \cite{X00}. Yuan proved that a Hom-GD-algebra is equivalent to a Hom-Lie conformal algebra of degree 2 \cite{Y14}. Later, the cited result was generalized to the super case by Yuan, Chen, and He \cite{YCH17}. 
Extending structures for GD-algebras were studied by Wen and Hong \cite{WH24}. It seems that one of the first attempts at obtaining the algebraic classification of $3$-dimensional GD-algebras was made in 2013 \cite{QZZ13}.\medskip 

On the other hand, Novikov--Poisson algebras are closely related to transposed Poisson algebras. Namely, if $(V,\cdot,\circ )$ is a Novikov--Poisson algebra, then $(V,\cdot,[-,-])$ is a transposed Poisson algebra, where $[x, y]=x\circ y-y\circ x$ \cite{BBGW23}. The notion of Novikov--Poisson algebras was introduced by Xu in 1996 \cite{X96}, and since then, they have been under active investigation. Xu used these algebras to construct new examples of simple Novikov algebras \cite{X97}. Zhelyabin and Tikhov established (via the Kantor double process) a connection between NP-algebras and Jordan superalgebras \cite{ZT08}. 
These ideas have been extended in some works by Zhelyabin and Zakharov \cite{Z13,Z19}. Free NP-algebras were studied by Bokut, Chen, and Zhang \cite{BCZ18}. In particular, they constructed a basis of free NP-algebras and proved a Poincaré--Birkhoff--Witt theorem: any NP-algebra is embeddable into its universal enveloping special Novikov--Poisson admissible algebra. 
Kolesnikov and Nesterenko proved that every Novikov--Poisson algebra over a field of characteristic zero can be embedded into a commutative conformal algebra with a derivation \cite{KN23}. Li and Wang established the bialgebra theory for NP-algebras \cite{LW25}. 
Zakharov obtained a classification of all finite-dimensional simple NP-algebras over a field of characteristic other than 2 \cite{Z23}. 
Pozhidaev, Zakharov, and Zhelyabin proved that every non-associative Novikov algebra can be equipped with a nontrivial structure of a Novikov--Poisson algebra \cite{PZZ25}. An algebraic classification of $2$-dimensional NP-algebras was given by Zhao, Bai, and Meng in \cite{ZCM04}; the full graph of degenerations of $2$-dimensional NP-algebras was given by Abdelwahab, Fernández Ouaridi, and Kaygorodov in \cite{AFK25}; and an early attempt at the classification of $3$-dimensional NP-algebras was made in \cite{Z15}. Recently, Zhang, Hao, Sun, and Chen obtained the complete algebraic classification of complex $3$-dimensional Novikov--Poisson and Gelfand--Dorfman algebras \cite{ZHSC,ZHSC-GD}.

 \section{\large Preliminaries: degenerations and    geometric classification}
 Given a complex vector space ${\mathbb V}$ of dimension $n$, the set of bilinear maps \begin{center}$\textrm{Bil}\big({\mathbb V} \times {\mathbb V}, {\mathbb V}\big) \cong \textrm{Hom}\big({\mathbb V} ^{\otimes2}, {\mathbb V}\big)\cong\big ({\mathbb V}^*\big)^{\otimes2} \otimes {\mathbb V}$
\end{center} is a vector space of dimension $n^3$. The set of pairs of bilinear maps 
\begin{center}$\textrm{Bil}\big({\mathbb V} \times {\mathbb V}, {\mathbb V}\big) \oplus \textrm{Bil}\big({\mathbb V}\times {\mathbb V}, {\mathbb V}\big) \cong \big(({\mathbb V}^*\big)^{\otimes2} \otimes {\mathbb V}) \oplus \big(({\mathbb V}^*)^{\otimes2} \otimes {\mathbb V}\big)$ \end{center} is a vector space of dimension $2n^3$. This vector space has the structure of an affine space $\mathbb{C}^{2n^3}$ in the following sense:
once a basis $e_1, \ldots, e_n$ of ${\mathbb V}$ is fixed, any pair $(\mu, \mu') \in \textrm{Bil}\big({\mathbb V} \times {\mathbb V}, {\mathbb V}\big)^2$ is determined by some parameters $c_{ij}^k, c_{ij}'^k \in \mathbb{C}$,  called {structural constants},  such that
$$\mu(e_i, e_j) = \sum_{k=1}^n c_{ij}^k e_k \textrm{ and } \mu'(e_i, e_j) = \sum_{k=1}^n c_{ij}'^k e_k$$
which corresponds to a point in the affine space $\mathbb{C}^{2n^3}$. Then a subset $\mathcal S$ of $\textrm{Bil}\big({\mathbb V} \times {\mathbb V}, {\mathbb V}\big)^2$ corresponds to an algebraic variety, i.e., a Zariski closed set, if there are some polynomial equations in variables $c_{ij}^k, c_{ij}'^k$ with zero locus equal to the set of structural constants of the bilinear pairs in $\mathcal S$. 

Given the identities defining Poisson-type algebras (in the present paper, Novikov--Poisson or Gelfand--Dorfman algebras), we can obtain a set of polynomial equations in the variables $c_{ij}^k, c_{ij}'^k$. This class of $n$-dimensional Poisson-type algebras is a variety. Denote it by $\mathcal{T}_{n}$.
Now, consider the following action of $\rm{GL}({\mathbb V})$ on ${\mathcal T}_{n}$:
$$\big(g*(\mu, \mu')\big)(x,y) \ := \ \big(g \mu (g^{-1} x, g^{-1} y),\  g \mu' (g^{-1} x, g^{-1} y)\big)$$
for $g\in\rm{GL}({\mathbb V})$, $(\mu, \mu')\in \mathcal{T}_{n}$ and for any $x, y \in {\mathbb V}$. Observe that the $\textrm{GL}({\mathbb V})$-orbit of $(\mu, \mu')$, denoted ${\mathcal O}\big((\mu, \mu')\big)$, contains all the structural constants of the bilinear pairs isomorphic to the Poisson-type algebra with structural constants $(\mu, \mu')$.

A geometric classification of a variety of algebras consists of describing the irreducible components of the variety. Recall that any affine variety can be represented uniquely as an irredundant finite union of its irreducible components.
Note that describing the irreducible components of  ${\mathcal{T}_{n}}$ gives us the rigid algebras of the variety, which are those bilinear pairs with an open $\textrm{GL}(\mathbb V)$-orbit. This is because a bilinear pair is rigid in a variety if and only if the closure of its orbit is an irreducible component of the variety. 
For this reason, the following notion is convenient. Denote by $\overline{{\mathcal O}\big((\mu, \mu')\big)}$ the closure of the orbit of $(\mu, \mu')\in{\mathcal{T}_{n}}$. For a parametric family of algebras ${\rm T}(*)=\big\{{\rm T}(\alpha):\alpha\in I\big\}$, we denote by $\overline{{\mathcal O}\big({\rm T}(*)\big)}=\overline{\bigcup_{\alpha\in I}{\mathcal O}\big({\rm T}(\alpha)\big)}$ the closure of the union of the orbits of all algebras of the family. When a family is denoted by its name with parameters 
$\big($for example, ${\rm NP}_{11}^{\alpha,\beta,\gamma}\big),$ 
we write $\overline{{\mathcal O}({\rm NP}_{11}^{\alpha,\beta,\gamma})}$ for the closure of the union of the orbits of all algebras of this family.

\begin{definition}
\rm Let ${\rm T} $ and ${\rm T}'$ be two $n$-dimensional Poisson-type algebras  corresponding to the variety $\mathcal{T}_{n}$ and $(\mu, \mu'),\  (\lambda,\lambda') \in \mathcal{T}_{n}$ be their representatives in the affine space, respectively. The algebra ${\rm T}$ is said to {degenerate}  to ${\rm T}'$, and we write ${\rm T} \to {\rm T} '$, if $(\lambda,\lambda')\in\overline{{\mathcal O}\big((\mu, \mu')\big)}$. If ${\rm T}  \not\cong {\rm T}'$, then we call it a  {proper degeneration}. Conversely, if $(\lambda,\lambda')\not\in\overline{{\mathcal O}\big((\mu, \mu')\big)}$ then  we say that ${{\rm T} }$ does not degenerate to ${{\rm T} }'$ and we write ${{\rm T} }\not\to {{\rm T} }'$.
\end{definition}
 \begin{remark}\label{rem-countable-orbit-closures}
Let ${\rm T}$ and ${\rm T}'$ be two nonisomorphic complex Poisson-type algebras represented by structures in ${\mathcal T}_n$. If all but at most countably many points of $\overline{\mathcal O({\rm T}')}$ belong to $\overline{\mathcal O({\rm T})}$, then
$
\overline{\mathcal O({\rm T}')}
\subsetneq
\overline{\mathcal O({\rm T})}.
$
Consequently, $\overline{\mathcal O({\rm T}')}$ is not an irreducible component of ${\mathcal T}_n$.
\end{remark} 

Furthermore, we have the following notion for a parametric family of algebras.

\begin{definition}
\rm
Let ${{\rm T} }(*)=\big\{{{\rm T} }(\alpha): {\alpha\in I}\big\}$ be a family of $n$-dimensional Poisson-type algebras  corresponding to ${\mathcal{T} }_n$ and let ${{\rm T} }'$ be another $n$-dimensional Poisson-type algebra. Suppose that ${{\rm T} }(\alpha)$ is represented by the structure $\big(\mu(\alpha),\mu'(\alpha)\big)\in{\mathcal{T} }_n$ for $\alpha\in I$ and ${{\rm T} }'$ is represented by the structure $(\lambda, \lambda')\in{\mathcal{T} }_n$. We say that the family ${{\rm T} }(*)$ {degenerates}   to ${{\rm T} }'$, and write ${{\rm T} }(*)\to {{\rm T} }'$, if $(\lambda,\lambda')\in\overline{ \big\{{\mathcal O}\big((\mu(\alpha),\mu'(\alpha))\big)\big\}_{\alpha\in I}}$.
Conversely, if $(\lambda,\lambda')\not\in\overline{\big\{{\mathcal O}\big((\mu(\alpha),\mu'(\alpha))\big)\big\}_{\alpha\in I}}$ then we call it a  {non-degeneration}, and we write ${{\rm T} }(*)\not\to {{\rm T} }'$.

\end{definition}

Observe that ${\rm T}'$ corresponds to an irreducible component of $\mathcal{T}_n$ (more precisely, $\overline{{\mathcal O}({\rm T}')}$ is an irreducible component) if and only if ${{\rm T} }\not\to {{\rm T} }'$ for any $n$-dimensional Poisson-type algebra ${\rm T}\not\cong{\rm T}'$ and ${{\rm T}(*) }\not\to {{\rm T} }'$ for any parametric family ${\rm T}(*)$ of $n$-dimensional  Poisson-type algebras which does not contain an algebra isomorphic to ${\rm T}'$. Similarly, the closure $\overline{{\mathcal O}({\rm T}(*))}$ of an irreducible parametric family ${\rm T}(*)$ is an irreducible component of $\mathcal{T}_n$ if and only if it is not properly contained in the orbit closure of an algebra or in the closure $\overline{{\mathcal O}({\rm T}'(*))}$ of another irreducible parametric family ${\rm T}'(*)$. In this case, we will use the following ideas to prove that a particular algebra (or family) corresponds to an irreducible component.

\noindent 
Firstly, since $\mathrm{dim}\,{\mathcal O}\big((\mu, \mu')\big) = n^2 - \mathrm{dim}\,\mathfrak{Der}({\rm T})$, if $ {\rm T} \to  {\rm T} '$ and  ${\rm T} \not\cong  {\rm T} '$, we have that $\dim\mathcal O(T)>\dim\mathcal O(T')$, where $\mathfrak{Der}( {\rm T} )$ denotes the Lie algebra of derivations of  ${\rm T} $. 

\noindent  Secondly, to prove degenerations, let ${{\rm T} }$ and ${{\rm T} }'$ be two  Poisson-type algebras represented by the structures $(\mu, \mu')$ and $(\lambda, \lambda')$ from ${{\mathcal T} }_n$, respectively. Let $c_{ij}^k, c_{ij}'^k$ be the structure constants of $(\lambda, \lambda')$ in a basis $e_1,\dots, e_n$ of ${\mathbb V}$. If there exist $n^2$ maps $a_i^j: \mathbb{C}^*\to \mathbb{C}$ such that $E_i(t)=\sum_{j=1}^na_i^j(t)e_j$ ($1\leq i \leq n$) form a basis of ${\mathbb V}$ for any $t\in\mathbb{C}^*$ and the structure constants $c_{ij}^k(t), c_{ij}'^k(t)$ of $(\mu, \mu')$ in the basis $E_1(t),\dots, E_n(t)$ satisfy $\lim\limits_{t\to 0}c_{ij}^k(t)=c_{ij}^k$ and $\lim\limits_{t\to 0}c_{ij}'^k(t)=c_{ij}'^k$, then ${{\rm T} }\to {{\rm T} }'$. In this case,  $E_1(t),\dots, E_n(t)$ is called a parametrized basis for ${{\rm T} }\to {{\rm T} }'$.

\noindent Thirdly, to prove non-degenerations, we may use the following lemma, see \cite{AFK25}  and the references therein. 

\begin{lemma}\label{main1}
Consider two Poisson-type algebras ${\rm T}$ and ${\rm T}'$. Suppose ${\rm T} \to {\rm T}'$. Let ${\rm C}$   be a Zariski closed set in ${\mathcal T}_n$ that is stable under the action of the invertible upper (lower) triangular matrices (with respect to the fixed basis $e_1,\dots,e_n$). Then if there is a representation $(\mu, \mu')$ of ${\rm T}$ in ${\rm C}$, then there is a representation $(\lambda, \lambda')$ of ${\rm T}'$ in ${\rm C}$.
\end{lemma}

To apply Lemma \ref{main1}, we will give the explicit definition of the appropriate stable Zariski closed set  ${\rm C}$ in terms of the variables $c_{ij}^k, c_{ij}'^k$ in each case: the set ${\rm C}$ consists of all structures whose structure constants satisfy the listed relations.

\begin{remark}
\label{redbil}

Moreover, let ${{\rm T} }$ and ${{\rm T} }'$ be two  Poisson-type algebras represented by the structures $(\mu, \mu')$ and $(\lambda, \lambda')$ from ${\mathcal T}_n$. Suppose ${\rm T}\to{\rm T}'$. Then if $\mu, \mu', \lambda, \lambda'$ represent algebras ${\rm T}_{0}, {\rm T}_{1}, {\rm T}'_{0}, {\rm T}'_{1}$ in the affine space $\mathbb{C}^{n^3}$ of algebras with a single multiplication, respectively, we have ${\rm T}_{0}\to {\rm T}'_{0}$ and $ {\rm T}_{1}\to {\rm T}'_{1}$.
So, for example, $(0, \mu)$ cannot degenerate to $(\lambda, 0)$ unless $\lambda=0$. 

\end{remark}

\noindent Fourthly, to prove ${{\rm T} }(*)\to {{\rm T} }'$, suppose that ${{\rm T} }(\alpha)$ is represented by the structure $\big(\mu(\alpha),\mu'(\alpha)\big)\in{\mathcal{T} }_n$ for $\alpha\in I$ and ${{\rm T} }'$ is represented by the structure $(\lambda, \lambda')\in{\mathcal{T} }_n$. Let $c_{ij}^k, c_{ij}'^k$ be the structure constants of $(\lambda, \lambda')$ in a basis  $e_1,\dots, e_n$ of ${\mathbb V}$. If there is a pair of maps $(f, (a_i^j))$, where $f:\mathbb{C}^*\to I$ and $a_i^j:\mathbb{C}^*\to \mathbb{C}$ are such that $E_i(t)=\sum_{j=1}^na_i^j(t)e_j$ ($1\le i\le n$) form a basis of ${\mathbb V}$ for any  $t\in\mathbb{C}^*$ and the structure constants $c_{ij}^k(t), c_{ij}'^k(t)$ of $\big(\mu\big(f(t)\big),\mu'\big(f(t)\big)\big)$ in the basis $E_1(t),\dots, E_n(t)$ satisfy $\lim\limits_{t\to 0}c_{ij}^k(t)=c_{ij}^k$ and $\lim\limits_{t\to 0}c_{ij}'^k(t)=c_{ij}'^k$, then ${{\rm T} }(*)\to {{\rm T} }'$. In this case,  $E_1(t),\dots, E_n(t)$ and $f(t)$ are called a parametrized basis and a parametrized index for ${{\rm T} }(*)\to {{\rm T} }'$, respectively.

\noindent Fifthly, to prove ${{\rm T} }(*)\not \to {{\rm T} }'$, we may use the following analog of Lemma \ref{main1} for parametric families, see \cite{afm} and the references therein.

\begin{lemma}\label{main2}
Consider a family of Poisson-type algebras ${\rm T}(*)$ and a  Poisson-type   algebra ${\rm T}'$. Suppose ${\rm T}(*) \to {\rm T}'$. Let ${\rm C}$  be a Zariski closed set in $\mathcal{T}_n$ that is stable under the action of the invertible upper (lower) triangular matrices (with respect to the fixed basis $e_1,\dots,e_n$). Then if there is a representation $\big(\mu(\alpha), \mu'(\alpha)\big)$ of ${\rm T}(\alpha)$ in ${\rm C}$ for every $\alpha\in I$, then there is a representation $(\lambda, \lambda')$ of ${\rm T}'$ in ${\rm C}$.
\end{lemma}

\noindent Finally, the following remark simplifies the geometric problem.

\begin{remark}\label{rem20}
 Let $(\mu, \mu')$ and $(\lambda, \lambda')$ represent two Poisson-type algebras. Suppose $(\lambda, 0)\not\in\overline{{\mathcal O}\big((\mu, 0)\big)}$, $($resp., $(0, \lambda')\not\in\overline{{\mathcal O}\big((0, \mu')\big)}$$)$, then $(\lambda, \lambda')\not\in\overline{{\mathcal O}\big((\mu, \mu')\big)}$. As we construct the classification of Poisson-type algebras from a certain class of algebras with a single multiplication which remains unchanged, this remark becomes very useful.  
\end{remark}

\subsection{\normalsize The geometric classification of 
$3$-dimensional Novikov algebras}

Beneš and Burde studied degenerations between complex $3$-dimensional Novikov algebras and the orbit closures in the variety of these algebras \cite{BB14}. 
As it follows from their results,  every $3$-dimensional Novikov algebra lies in the orbit closure of one of the algebras ${\rm N}_{01}-{\rm N}_{11}^{\alpha,\beta},$ where
\begin{longtable}{lclll}
    $\rm{N}_{01}$&$:$&$e_1\circ e_{1}=e_1$&$ e_{2}\circ e_2=e_2$&$ e_3\circ e_{3}=e_3$\\
    $\rm{N}_{02}$&$:$&$e_1\circ e_{1}=e_1$&$e_1\circ e_{2}= e_2+e_3$&$e_1\circ e_3=e_3$\\
                             &&$e_2\circ e_1=e_2$& $e_2\circ e_2=e_3$ & $e_3 \circ e_1=e_3$\\
    $\rm{N}_{03}$&$:$&$e_1\circ e_{1}=-e_1+e_2$&$ e_{2}\circ e_1=-e_2$&$e_3\circ e_{3}=e_3$\\
    $\rm{N}_{04}$&$:$&$e_1\circ e_{1}=-e_1+e_3$&$e_1\circ e_{3}= e_2$&$e_2\circ e_1=-e_2$\\
                             &&$e_3\circ e_1=-e_3$\\
    
    $\rm{N}_{05}$&$:$&$e_1\circ e_{1}=-\frac{1}{2}e_1+e_3$&$e_1\circ e_{2}= \frac{1}{2}e_2$&$e_1\circ e_3=e_2$\\
                             &&$e_2\circ e_1=-\frac{1}{2}e_2$& $e_3\circ e_1=e_2-\frac{1}{2}e_3$\\
    $\rm{N}_{06}$&$:$&$e_1\circ e_{1}=-e_1+e_2$&$e_1\circ e_{3}=- \frac{1}{2}e_3$&$e_2\circ e_1=-e_2$\\
                             &&$e_3\circ e_1=-e_3$& $e_3\circ e_3=e_2$ \\
    $\rm{N}_{07}^\alpha$&$:$&$e_1\circ e_{1}=\alpha e_1$&$e_1\circ e_{2}=(\alpha+1) e_2$&$e_2\circ e_1=\alpha e_2$\\
                             &&$e_3\circ e_3=e_3$\\
    $\rm{N}_{08}^\alpha$&$:$&$e_1\circ e_{1}=\alpha e_1$&$e_1\circ e_{2}=(\alpha+1) e_2$&$e_1\circ e_3=e_2+(\alpha+1)e_3$\\
                             &&$e_2\circ e_1=\alpha e_2$& $e_3\circ e_1=\alpha e_3$\\
    $\rm{N}_{09}^\alpha$&$:$&$e_1\circ e_{1}=-e_1+e_2$&$e_1\circ e_{3}= (\alpha-1)e_3$&$e_2\circ e_1=-e_2$\\
                             &&$e_3\circ e_1=-e_3$\\
    $\rm{N}_{10}^\alpha$&$:$&$e_1\circ e_{1}=\alpha e_1$&$e_1\circ e_{2}=(\alpha+1) e_2$&$e_1\circ e_3=(\alpha+\frac{1}{2})e_3$\\
                             &&$e_2\circ e_1=\alpha e_2$& $e_3\circ e_1=\alpha e_3$ & $e_3 \circ e_3=e_2$\\
    $\rm{N}_{11}^{\alpha, \beta}$&$:$&$e_1\circ e_{1}=\beta e_1$&$e_1\circ e_{2}=(\beta+1) e_2$&$e_1\circ e_3=(\alpha+\beta)e_3$\\
                             &&$e_2\circ e_1=\beta e_2$& $e_3\circ e_1=\beta e_3$\\
\end{longtable}
 
However, for our goal, we will need a stronger result.
Namely, we need the following proposition.

\begin{proposition}\label{propN}
The variety of complex $3$-dimensional Novikov algebras has dimension $9$ and it has $4$ irreducible components defined by
\begin{center}
        $\mathcal{C}_{1}=\overline{{\mathcal O}( {\rm N}_{01})},$ \
        $\mathcal{C}_{2}=\overline{{\mathcal O}( {\rm N}_{07}^{\alpha})},$ \
        $\mathcal{C}_{3}=\overline{{\mathcal O}( {\rm N}_{10}^{\alpha})},$ \ and
        $\mathcal{C}_{4}=\overline{{\mathcal O}( {\rm N}_{11}^{\alpha,\beta})}.$
\end{center}
In particular, there is only one rigid algebra  ${\rm N}_{01}$.
\end{proposition}

\begin{proof}

The dimensions of the orbit closures of the important algebras are as follows:
\begin{longtable}{lllllllll}
$\dim  \mathcal{C}_{1}$ & $\!\!=\!\!$ & 
$\dim  \mathcal{C}_{2}$ & $\!\!=\!\!$ & 
$\dim  \mathcal{C}_{3}$ & $\!\!=\!\!$ & 
$\dim  \mathcal{C}_{4}$ & $\!\!=\!\!$ & $9.$ \\
\end{longtable}

The remaining algebras of the list ${\rm N}_{01}-{\rm N}_{11}^{\alpha,\beta}$ lie in $\mathcal{C}_2\cup\mathcal{C}_3\cup\mathcal{C}_4$, as the following degenerations show.

\begin{longtable}{lcl|lcl}
${\rm N}_{10}^{\frac{1}{t}}$ & $\xrightarrow{\big(te_1,\ e_3+\frac{2}{t}e_2,\ e_2\big)}$ & ${\rm N}_{02}$ &
${\rm N}_{07}^{-1+t}$ & $\xrightarrow{\big(e_1+e_2,\ te_2,\ e_3\big)}$ & ${\rm N}_{03}$\\[0.5ex]
\hline
${\rm N}_{08}^{-1+t}$ & $\xrightarrow{\big(e_1+e_3,\ te_2,\ e_2+te_3\big)}$ & ${\rm N}_{04}$ &
${\rm N}_{10}^{-\frac{1}{2}+t}$ & $\xrightarrow{\big(e_1-\frac{2}{2t+1}e_2+e_3,\ te_2,\ te_3\big)}$ & ${\rm N}_{05}$\\[0.5ex]
\hline
${\rm N}_{10}^{-1+t}$ & $\xrightarrow{\big(e_1+\frac{1}{t}e_2,\ e_2,\ e_3\big)}$ & ${\rm N}_{06}$ &
${\rm N}_{11}^{1+t,\alpha}$ & $\xrightarrow{\big(e_1,\ e_3,\ \frac{1}{t}(e_3-e_2)\big)}$ & ${\rm N}_{08}^{\alpha}$\\[0.5ex]
\hline
\multicolumn{6}{c}{${\rm N}_{11}^{\alpha-t,-1+t}\qquad  \xrightarrow{\big(e_1+e_2,\ te_2,\ e_3\big)} \qquad {\rm N}_{09}^{\alpha}$}\\[0.5ex]
\hline
\end{longtable}
\end{proof}

 
\section{\large The geometric classification of Novikov--Poisson  algebras}

\begin{definition}
    A Novikov algebra is a vector space with a bilinear product $\circ$ satisfying 
    \begin{longtable}{rcl}
$(x \circ y) \circ z - x \circ (y \circ z)$&$ =$&$ (y \circ x) \circ z - y \circ (x \circ z),$\\
$(x \circ y) \circ z $&$=$&$ (x \circ z) \circ y.$
\end{longtable}
\end{definition}

\begin{definition}[see \cite{X96}]
A Novikov--Poisson algebra is a triple $({\rm A}, \cdot, \circ)$ in which $\cdot$ and $\circ$ are
bilinear operations on ${\rm A}$ such that
$({\rm A}, \cdot)$ is a commutative associative algebra;
 $({\rm A}, \circ)$ is a Novikov algebra;
and the following compatibility identities hold
\begin{longtable}{rcl}
$(x \cdot y) \circ z$&$ =$&$ x \cdot (y \circ z),$\\
$(x \circ y) \cdot z - (y \circ x) \cdot z$&$ =$&$ x \circ (y \cdot z) - y \circ (x \cdot z).$
\end{longtable}
\end{definition}

 \subsection{\normalsize The algebraic classification of $3$-dimensional Novikov--Poisson algebras}

\begin{theorem}[see \cite{ZHSC}]
Let $(\rm{A},\cdot,\circ)$ be a complex $3$-dimensional Novikov--Poisson algebra with nonzero multiplication $\cdot$. Then $(\rm{A},\cdot,\circ)$ is isomorphic to one of the following algebras:
\footnotesize

\begin{longtable}{lcl}

 $\mathrm{NP}_{01}^{\alpha,\beta}$&$:$&$
\begin{cases}
		e_1 \cdot e_3 = e_2; \\
		e_1 \circ e_1 = e_1, \quad e_1 \circ e_2 = -e_2, \quad e_1 \circ e_3 = e_1, \quad e_2 \circ e_1 = e_2, \\ 
        e_2 \circ e_3 = e_2, \quad
		e_3 \circ e_1 = \alpha e_2+e_3, \quad e_3 \circ e_2 = -e_2, \quad e_3 \circ e_3 = \beta e_2+e_3
	\end{cases}
	$\\

$\mathrm{NP}_{02}^{\alpha}$&$:$&
	$
	\begin{cases}
		e_1 \cdot e_3 = e_2; \\
		e_1 \circ e_1 = e_1, \quad e_1 \circ e_2 = -e_2, \quad e_2 \circ e_1 = e_2, \quad 
        e_3 \circ e_1 = \alpha e_2+e_3, \quad
		e_3 \circ e_3 = e_2
	\end{cases}
	$\\

	$\mathrm{NP}_{03}^{\alpha}$&$:$&
	$
	\begin{cases}
		e_1 \cdot e_3 = e_2; \\
		e_1 \circ e_1 = e_1, \quad e_1 \circ e_2 = -e_2, \quad e_2 \circ e_1 = e_2, \quad e_3 \circ e_1 = \alpha e_2+e_3
	\end{cases}
	$\\

$\mathrm{NP}_{04}^{\alpha,\beta,\gamma}$&$:$&
	$
	\begin{cases}
		e_1 \cdot e_3 = e_2; \\
		e_1 \circ e_1 = e_2, \quad e_1 \circ e_3 = \alpha e_2, \quad e_3 \circ e_1 = \beta e_2, \quad e_3 \circ e_3 = \gamma e_2
	\end{cases}
	$\\

$\mathrm{NP}_{05}^{\alpha,\beta}$&$:$&
	$
	\begin{cases}
		e_1 \cdot e_3 = e_2; \\
		e_1 \circ e_3 = \alpha e_2, \quad e_3 \circ e_1 = \beta e_2
	\end{cases}
	$\\

$\mathrm{NP}_{06}^{\alpha,\beta} $&$:$&
	$
	\begin{cases}
		e_1 \cdot e_1 = e_2, \quad e_1 \cdot e_2 = e_3; \\
		e_1 \circ e_1 = \alpha e_1, \quad e_1 \circ e_2 = e_2+\beta e_3, \quad e_1 \circ e_3 = (2-\alpha) e_3, \\
		e_2 \circ e_1 = \alpha e_2, \quad e_2 \circ e_2 = e_3, \quad e_3 \circ e_1 = \alpha e_3
	\end{cases}$\\

$\mathrm{NP}_{07}^{\alpha,\beta}$&$:$&
	$
	\begin{cases}
		e_1 \cdot e_1 = e_2, \quad e_1 \cdot e_2 = e_3; \\
		e_1 \circ e_1 = 2e_1+\alpha e_3, \quad e_1 \circ e_2 = e_2+\beta e_3, \quad e_2 \circ e_1 = 2e_2, \
		e_2 \circ e_2 = e_3, \quad e_3 \circ e_1 = 2e_3
	\end{cases}
	$\\

$\mathrm{NP}_{08}^{\alpha}$&$:$&
	$
	\begin{cases}
		e_1 \cdot e_1 = e_2, \quad e_1 \cdot e_2 = e_3 ;\\
		e_1 \circ e_1 = e_1+\alpha e_2, \quad e_1 \circ e_3 = -e_3, \quad e_2 \circ e_1 = e_2+\alpha e_3, \quad e_3 \circ e_1 = e_3
	\end{cases}
	$\\

$\mathrm{NP}_{09}^{\alpha,\beta}$&$:$&
	$
	\begin{cases}
		e_1 \cdot e_1 = e_2, \quad e_1 \cdot e_2 = e_3; \\
		e_1 \circ e_1 = \alpha e_2, \quad e_1 \circ e_2 = \beta e_3, \quad e_2 \circ e_1 = \alpha e_3
	\end{cases}
	$\\

$\mathrm{NP}_{10}^{\alpha}$&$:$&
	$
	\begin{cases}
		e_1 \cdot e_1 = e_2, \quad e_1 \cdot e_2 = e_3; \\
		e_1 \circ e_1 = \alpha e_2+e_3, \quad e_1 \circ e_2 = \alpha e_3, \quad e_2 \circ e_1 = \alpha e_3
	\end{cases}
	$\\

$\mathrm{NP}_{11}^{\alpha,\beta,\gamma}$&$:$&
	$
	\begin{cases}
		e_1 \cdot e_3 = e_1, \quad e_2 \cdot e_3 = e_2, \quad e_3 \cdot e_3 = e_3 ;\\
		e_1 \circ e_3 = \alpha e_1, \quad e_2 \circ e_3 = \alpha e_2, \quad e_3 \circ e_1 = e_2, \quad  e_3 \circ e_2 = \beta e_1+\gamma e_2, \quad 
		e_3 \circ e_3 = e_1+\alpha e_3
	\end{cases}
	$\\

$\mathrm{NP}_{12}^{\alpha,\beta,\gamma}$&$:$&
	$
	\begin{cases}
		e_1 \cdot e_3 = e_1, \quad e_2 \cdot e_3 = e_2, \quad e_3 \cdot e_3 = e_3; \\
		e_1 \circ e_3 = \alpha e_1, \quad e_2 \circ e_3 = \alpha e_2, \quad e_3 \circ e_1 = \beta e_1,  \quad 
        e_3 \circ e_2 = e_1+\gamma e_2, \quad 
		e_3 \circ e_3 = e_1+\alpha e_3
	\end{cases}
	$\\

$\mathrm{NP}_{13}^{\alpha,\beta}$&$:$&
	$
	\begin{cases}
		e_1 \cdot e_3 = e_1, \quad e_2 \cdot e_3 = e_2, \quad e_3 \cdot e_3 = e_3; \\
		e_1 \circ e_3 = \alpha e_1, \quad e_2 \circ e_3 = \alpha e_2, \quad e_3 \circ e_1 = \beta e_1, \quad e_3 \circ e_2 = \beta e_2,\quad
		e_3 \circ e_3 = e_1+\alpha e_3
	\end{cases}
	$\\

$\mathrm{NP}_{14}^{\alpha,\beta, \gamma}$&$:$&
	$
	\begin{cases}
		e_1 \cdot e_3 = e_1, \quad e_2 \cdot e_3 = e_2, \quad e_3 \cdot e_3 = e_3; \\
		e_1 \circ e_3 = \alpha e_1, \quad e_2 \circ e_3 = \alpha e_2, \quad e_3 \circ e_1 = \beta e_2, \quad 
		e_3 \circ e_2 = e_1+\gamma e_2, \quad e_3 \circ e_3 = \alpha e_3
	\end{cases}
	$\\

$\mathrm{NP}_{15}^{\alpha,\beta}$&$:$&
	$
	\begin{cases}
		e_1 \cdot e_3 = e_1, \quad e_2 \cdot e_3 = e_2, \quad e_3 \cdot e_3 = e_3; \\
		e_1 \circ e_3 = \alpha e_1, \quad e_2 \circ e_3 = \alpha e_2, \quad e_3 \circ e_1 = \beta e_1, \quad 
		e_3 \circ e_2 = \beta e_2, \quad e_3 \circ e_3 = \alpha e_3
	\end{cases}
	$\\

$\mathrm{NP}_{16}^{\alpha,\beta}$&$:$&
	$
	\begin{cases}
		e_1 \cdot e_3 = e_2, \quad e_2 \cdot e_3 = e_2, \quad e_3 \cdot e_3 = e_3; \\
		e_1 \circ e_1 = -e_1+e_2, \quad e_1 \circ e_3 = \alpha e_2, \quad e_2 \circ e_3 = \alpha e_2, \\  
        e_3 \circ e_1 = \beta e_2, \quad
		e_3 \circ e_2 = \beta e_2, \quad e_3 \circ e_3 = e_2+\alpha e_3
	\end{cases}
	$\\

	$\mathrm{NP}_{17}^{\alpha,\beta}$&$:$&
	$
	\begin{cases}
		e_1 \cdot e_3 = e_2, \quad e_2 \cdot e_3 = e_2, \quad e_3 \cdot e_3 = e_3; \\
		e_1 \circ e_3 = \alpha e_2, \quad e_2 \circ e_3 = \alpha e_2, \quad e_3 \circ e_1 = \beta e_2, \quad e_3 \circ e_2 = \beta e_2, \quad
		e_3 \circ e_3 = e_2+\alpha e_3
	\end{cases}
	$\\

	$\mathrm{NP}_{18}^{\alpha,\beta}$&$:$&
	$
	\begin{cases}
		e_1 \cdot e_3 = e_2, \quad e_2 \cdot e_3 = e_2, \quad e_3 \cdot e_3 = e_3; \\
		e_1 \circ e_1 = -e_1+e_2, \quad e_1 \circ e_3 = \alpha e_2, \quad e_2 \circ e_3 = \alpha e_2, \quad 		e_3 \circ e_1 = \beta e_2, \\  e_3 \circ e_2 = \beta e_2, \quad e_3 \circ e_3 = \alpha e_3
	\end{cases}
	$\\

	$\mathrm{NP}_{19}^{\alpha,\beta}$&$:$&
	$
	\begin{cases}
		e_1 \cdot e_3 = e_2, \quad e_2 \cdot e_3 = e_2, \quad e_3 \cdot e_3 = e_3; \\
		e_1 \circ e_3 = \alpha e_2, \quad e_2 \circ e_3 = \alpha e_2, \quad e_3 \circ e_1 = \beta e_2, \quad 
		e_3 \circ e_2 = \beta e_2, \quad e_3 \circ e_3 = \alpha e_3
	\end{cases}
	$\\

	$\mathrm{NP}_{20}^{\alpha,\beta,\gamma}$&$:$&
	$
	\begin{cases}
		e_1 \cdot e_1 = e_2, \quad e_1 \cdot e_3 = e_1, \quad e_2 \cdot e_3 = e_2, \quad e_3 \cdot e_3 = e_3; \\
        
		e_1 \circ e_1 = (\frac{\alpha+\beta}{2}) e_2, \quad e_1 \circ e_3 = \beta e_1+e_2, \quad e_2 \circ e_3 = \beta e_2, \\
		e_3 \circ e_1 = (\frac{\alpha+\beta}{2}) e_1+\gamma e_2, \quad e_3 \circ e_2 = \alpha e_2, \quad e_3 \circ e_3 = e_1+\beta e_3
	\end{cases}
	$\\

	$\mathrm{NP}_{21}^{\alpha,\beta}$&$:$&
	$
	\begin{cases}
		e_1 \cdot e_1 = e_2, \quad e_1 \cdot e_3 = e_1, \quad e_2 \cdot e_3 = e_2, \quad e_3 \cdot e_3 = e_3; \\
		e_1 \circ e_1 = (\frac{\alpha+\beta}{2}) e_2, \quad e_1 \circ e_3 = \beta e_1, \quad e_2 \circ e_3 = \beta e_2, \\
		e_3 \circ e_1 = (\frac{\alpha+\beta}{2}) e_1, \quad e_3 \circ e_2 = \alpha e_2, \quad e_3 \circ e_3 = e_2+\beta e_3
	\end{cases}
	$\\

	$\mathrm{NP}_{22}^{\alpha,\beta} $&$:$&
	$
	\begin{cases}
		e_1 \cdot e_1 = e_2, \quad e_1 \cdot e_3 = e_1, \quad e_2 \cdot e_3 = e_2, \quad e_3 \cdot e_3 = e_3; \\
		e_1 \circ e_1 = \alpha e_2, \quad e_1 \circ e_3 = \alpha e_1, \quad e_2 \circ e_3 = \alpha e_2, \\ 
        e_3 \circ e_1 = \alpha e_1+\beta e_2, \quad 
		e_3 \circ e_2 = \alpha e_2, \quad e_3 \circ e_3 = e_2+\alpha e_3
	\end{cases}
	$\\

	$\mathrm{NP}_{23}^{\alpha,\beta}$&$:$&
	$
	\begin{cases}
		e_1 \cdot e_1 = e_2, \quad e_1 \cdot e_3 = e_1, \quad e_2 \cdot e_3 = e_2, \quad e_3 \cdot e_3 = e_3; \\
		e_1 \circ e_1 = (\frac{\alpha+\beta}{2}) e_2, \quad e_1 \circ e_3 = \alpha e_1, \quad e_2 \circ e_3 = \alpha e_2, \\
		e_3 \circ e_1 = (\frac{\alpha+\beta}{2}) e_1, \quad e_3 \circ e_2 = \beta e_2, \quad e_3 \circ e_3 = \alpha e_3
	\end{cases}
	$\\

	$\mathrm{NP}_{24}^{\alpha}$&$:$&
	$
	\begin{cases}
		e_1 \cdot e_1 = e_2, \quad e_1 \cdot e_3 = e_1, \quad e_2 \cdot e_3 = e_2, \quad e_3 \cdot e_3 = e_3; \\
		e_1 \circ e_1 = \alpha e_2, \quad e_1 \circ e_3 = \alpha e_1, \quad e_2 \circ e_3 = \alpha e_2, \\ 
        e_3 \circ e_1 = \alpha e_1+e_2, \quad
		e_3 \circ e_2 = \alpha e_2, \quad e_3 \circ e_3 = \alpha e_3
	\end{cases}
	$\\

$\mathrm{NP}_{25}^{\alpha,\beta,\gamma}$&$:$&
	$
	\begin{cases}
		e_1 \cdot e_1 = e_1, \quad e_2 \cdot e_2 = e_2, \quad e_3 \cdot e_3 = e_3; \\
		e_1 \circ e_1 = \alpha e_1, \quad e_2 \circ e_2 = \beta e_2, \quad e_3 \circ e_3 = \gamma e_3
	\end{cases}
$\\

	$\mathrm{NP}_{26}^{\alpha,\beta,\gamma}$&$:$&
	$
	\begin{cases}
		e_1 \cdot e_1 = e_1, \quad e_2 \cdot e_2 = e_2, \quad e_1 \cdot e_3 = e_3; \\
		e_1 \circ e_1 = \alpha e_1 + e_3, \quad e_1 \circ e_3 = \beta e_3, \quad e_2 \circ e_2 = \gamma e_2, \quad e_3 \circ e_1 = \alpha e_3
	\end{cases}
	$\\

	$\mathrm{NP}_{27}^{\alpha,\beta,\gamma}$&$:$&
	$
	\begin{cases}
		e_1 \cdot e_1 = e_1, \quad e_2 \cdot e_2 = e_2, \quad e_1 \cdot e_3 = e_3 ;\\
		e_1 \circ e_1 = \alpha e_1, \quad e_1 \circ e_3 = \beta e_3, \quad e_2 \circ e_2 = \gamma e_2, \quad e_3 \circ e_1 = \alpha e_3
	\end{cases}
	$\\

	$\mathrm{NP}_{28}^{\alpha,\beta}$&$:$&
	$
	\begin{cases}
		e_1 \cdot e_1 = e_1, \quad e_2 \cdot e_2 = e_2; \\
		e_1 \circ e_1 = \alpha e_1, \quad e_2 \circ e_2 = \beta e_2, \quad e_3 \circ e_3 = e_3
	\end{cases}
	$\\

	$\mathrm{NP}_{29}^{\alpha,\beta}$&$:$&
	$
	\begin{cases}
		e_1 \cdot e_1 = e_1, \quad e_2 \cdot e_2 = e_2 ;\\
		e_1 \circ e_1 = \alpha e_1, \quad e_2 \circ e_2 = \beta e_2
	\end{cases}
	$\\

	$\mathrm{NP}_{30}^{\alpha,\beta}$&$:$&
	$
	\begin{cases}
		e_1 \cdot e_1 = e_1, \quad e_2 \cdot e_2 = e_3; \\
		e_1 \circ e_1 = \alpha e_1, \quad e_2 \circ e_2 = \beta e_2, \quad e_2 \circ e_3 = e_3, \quad e_3 \circ e_2 = \beta e_3
	\end{cases}
	$\\

	$\mathrm{NP}_{31}^{\alpha,\beta}$&$:$&
	$
	\begin{cases}
		e_1 \cdot e_1 = e_1, \quad e_2 \cdot e_2 = e_3 ;\\
		e_1 \circ e_1 = \alpha e_1, \quad e_2 \circ e_2 = e_2 + \beta e_3, \quad e_3 \circ e_2 = e_3
	\end{cases}
	$\\

	$\mathrm{NP}_{32}^{\alpha,\beta}$&$:$&
	$
	\begin{cases}
		e_1 \cdot e_1 = e_1, \quad e_2 \cdot e_2 = e_3; \\
		e_1 \circ e_1 = \alpha e_1, \quad e_2 \circ e_2 = \beta e_3
	\end{cases}
	$\\

	$\mathrm{NP}_{33}^{\alpha,\beta}$&$:$&
	$
	\begin{cases}
		e_1 \cdot e_1 = e_1; \\
		e_1 \circ e_1 = \alpha e_1, \quad e_2 \circ e_2 = \beta e_2, \quad e_2 \circ e_3 = e_3, \quad e_3 \circ e_2 = \beta e_3
	\end{cases}
	$\\

$\mathrm{NP}_{34}^{\alpha}$&$:$&
$
\begin{cases}
e_1 \cdot e_1 = e_1; \\
e_1 \circ e_1 = \alpha e_1, \quad e_2 \circ e_2 = e_2+e_3, \quad e_3 \circ e_2 = e_3
\end{cases}
$\\

$\mathrm{NP}_{35}^{\alpha}$&$:$&
$
\begin{cases}
e_1 \cdot e_1 = e_1; \\
e_1 \circ e_1 = \alpha e_1, \quad e_2 \circ e_2 = e_2, \quad e_3 \circ e_2 = e_3
\end{cases}
$\\

$\mathrm{NP}_{36}^{\alpha}$&$:$&
$
\begin{cases}
e_1 \cdot e_1 = e_1; \\
e_1 \circ e_1 = \alpha e_1, \quad e_2 \circ e_2 = e_3
\end{cases}
$\\

$\mathrm{NP}_{37}^{\alpha}$&$:$&
$
\begin{cases}
e_1 \cdot e_1 = e_1; \\
e_1 \circ e_1 = \alpha e_1
\end{cases}
$\\

$\mathrm{NP}_{38}^{\alpha}$&$:$&
$
\begin{cases}
e_1 \cdot e_1 = e_1; \\
e_1 \circ e_1 = \alpha e_1, \quad e_2 \circ e_2 = e_2, \quad e_3 \circ e_3 = e_3
\end{cases}
$\\

$\mathrm{NP}_{39}^{\alpha}$&$:$&
$
\begin{cases}
e_1 \cdot e_1 = e_1; \\
e_1 \circ e_1 = \alpha e_1, \quad e_2 \circ e_2 = e_2
\end{cases}
$\\

	$\mathrm{NP}_{40}^{\alpha,\beta} $&$:$&
	$
	\begin{cases}
		e_1 \cdot e_1 = e_2; \\
		e_1 \circ e_1 = e_1, \quad e_1 \circ e_2 = \alpha e_2, \quad e_1 \circ e_3 = \beta e_3, \quad e_2 \circ e_1 = e_2, \quad e_3 \circ e_1 = e_3
	\end{cases}
	$\\

	$\mathrm{NP}_{41}^{\alpha}$&$:$&
	$
	\begin{cases}
		e_1 \cdot e_1 = e_2; \\
		e_1 \circ e_1 = e_1, \quad e_1 \circ e_2 = \alpha e_2, \quad e_1 \circ e_3 = e_2 + \alpha e_3, \quad e_2 \circ e_1 = e_2, \quad e_3 \circ e_1 = e_3
	\end{cases}
	$\\

	$\mathrm{NP}_{42}^{\alpha,\beta} $&$:$&
	$
	\begin{cases}
		e_1 \cdot e_1 = e_2; \\
		e_1 \circ e_1 = e_1 + \alpha e_2, \quad e_1 \circ e_3 = \beta e_3, \quad e_2 \circ e_1 = e_2, \quad e_3 \circ e_1 = e_3
	\end{cases}
	$\\

	$\mathrm{NP}_{43}^{\alpha,\beta} $&$:$&
	$
	\begin{cases}
		e_1 \cdot e_1 = e_2; \\
		e_1 \circ e_1 = e_1, \quad e_1 \circ e_2 = \alpha e_2 + e_3, \quad e_1 \circ e_3 = \beta e_2, \quad e_2 \circ e_1 = e_2, \quad e_3 \circ e_1 = e_3
	\end{cases}
	$\\

	$\mathrm{NP}_{44}^{\alpha,\beta}$&$:$&
	$
	\begin{cases}
		e_1 \cdot e_1 = e_2; \\
		e_1 \circ e_1 = e_1 + \alpha e_2, \quad e_1 \circ e_2 = \beta e_2 + e_3, \quad e_2 \circ e_1 = e_2, \quad e_3 \circ e_1 = e_3
	\end{cases}
	$\\

	$\mathrm{NP}_{45}^{\alpha} $&$:$&
	$
	\begin{cases}
		e_1 \cdot e_1 = e_2; \\
		e_1 \circ e_1 = e_1 + e_3, \quad e_1 \circ e_2 = \alpha e_2, \quad e_2 \circ e_1 = e_2, \quad e_3 \circ e_1 = e_3
	\end{cases}
	$\\

	$\mathrm{NP}_{46}^{\alpha}$&$:$&
	$
	\begin{cases}
		e_1 \cdot e_1 = e_2; \\
		e_1 \circ e_1 = e_1 + e_3, \quad e_1 \circ e_3 = \alpha e_2, \quad e_2 \circ e_1 = e_2, \quad e_3 \circ e_1 = e_3
	\end{cases}
	$\\

	$\mathrm{NP}_{47}^{\alpha}$&$:$&
	$
	\begin{cases}
		e_1 \cdot e_1 = e_2; \\
		e_1 \circ e_1 = \alpha e_2, \quad e_1 \circ e_2 = e_3, \quad e_1 \circ e_3 = e_3
	\end{cases}
	$\\

	$\mathrm{NP}_{48}^{\alpha}$&$:$&
	$
	\begin{cases}
		e_1 \cdot e_1 = e_2;\\
		e_1 \circ e_2 = e_2, \quad e_1 \circ e_3 = \alpha e_3
	\end{cases}
	$\\

	$\mathrm{NP}_{49}
	$&$:$&$
	\begin{cases}
		e_1 \cdot e_1 = e_2; \\
		e_1 \circ e_2 = e_2, \quad e_1 \circ e_3 = e_2 + e_3
	\end{cases}
	$\\

	$\mathrm{NP}_{50}^{\alpha}$&$:$&
	$
	\begin{cases}
		e_1 \cdot e_1 = e_2; \\
		e_1 \circ e_1 = \alpha e_2, \quad e_1 \circ e_3 = e_3
	\end{cases}
	$\\

	$\mathrm{NP}_{51}^{\alpha}$&$:$&
	$
	\begin{cases}
		e_1 \cdot e_1 = e_2; \\
		e_1 \circ e_2 = \alpha e_2 + e_3, \quad e_1 \circ e_3 = e_2
	\end{cases}
	$\\

	$\mathrm{NP}_{52}$&$:$&
	$
	\begin{cases}
		e_1 \cdot e_1 = e_2; \\
		e_1 \circ e_1 = e_3, \quad e_1 \circ e_3 = e_2
	\end{cases}
	$\\

	$\mathrm{NP}_{53}$&$:$&
	$
	\begin{cases}
		e_1 \cdot e_1 = e_2; \\
		e_1 \circ e_3 = e_2
	\end{cases}
	$\\

	$\mathrm{NP}_{54}^{\alpha}$&$:$&
	$
	\begin{cases}
		e_1 \cdot e_1 = e_2; \\
		e_1 \circ e_1 = \alpha e_2, \quad e_1 \circ e_2 = e_3
	\end{cases}
	$\\

	$\mathrm{NP}_{55}$&$:$&
	$
	\begin{cases}
		e_1 \cdot e_1 = e_2; \\
		e_1 \circ e_1 = e_3, \quad e_1 \circ e_2 = e_2
	\end{cases}
	$\\

	$\mathrm{NP}_{56}$&$:$&
	$
	\begin{cases}
		e_1 \cdot e_1 = e_2; \\
		e_1 \circ e_1 = e_3
	\end{cases}
	$\\

	$\mathrm{NP}_{57}^{\alpha}$&$:$&
	$
	\begin{cases}
		e_1 \cdot e_1 = e_2; \\
		e_1 \circ e_1 = \alpha e_2
	\end{cases}
	$\\

	$\mathrm{NP}_{58}^{\alpha}$&$:$&
	$
	\begin{cases}
		e_1 \cdot e_1 = e_2; \\
		e_1 \circ e_1 = \alpha e_2, \quad e_1 \circ e_3 = e_1+e_2, \quad e_2 \circ e_3 = e_2, \quad e_3 \circ e_2 = -e_2, \quad e_3 \circ e_3 = e_3
	\end{cases}
	$\\

	$\mathrm{NP}_{59}^{\alpha}$&$:$&
	$
	\begin{cases}
		e_1 \cdot e_1 = e_2; \\
		e_1 \circ e_1 = \alpha e_2, \quad e_1 \circ e_3 = e_1, \quad e_2 \circ e_3 = e_2, \quad e_3 \circ e_2 = -e_2, \quad e_3 \circ e_3 = e_3
	\end{cases}
	$\\

	$\mathrm{NP}_{60}^{\alpha} $&$:$&
	$
	\begin{cases}
		e_1 \cdot e_1 = e_2; \\
		e_1 \circ e_1 = \alpha e_1, \quad e_1 \circ e_2 = (2-\alpha) e_2, \quad e_1 \circ e_3 = e_3, \\ e_2 \circ e_1 = \alpha e_2,\quad  e_3 \circ e_1 = \alpha e_3, \quad e_3 \circ e_3 = e_2
	\end{cases}
	$\\

	$\mathrm{NP}_{61}^{\alpha}$&$:$&
	$
	\begin{cases}
		e_1 \cdot e_1 = e_2; \\
		e_1 \circ e_1 = e_1, \quad e_1 \circ e_2 = e_2, \quad e_1 \circ e_3 = \alpha e_2+e_3, \\ 
        e_2 \circ e_1 = e_2, \quad  e_3 \circ e_1 = e_3, \quad e_3 \circ e_3 = e_2
	\end{cases}
	$\\

	$\mathrm{NP}_{62}^{\alpha}$&$:$&
	$
	\begin{cases}
		e_1 \cdot e_1 = e_2; \\
		e_1 \circ e_1 = 2e_1+\alpha e_2, \quad e_1 \circ e_3 = e_3, \quad e_2 \circ e_1 = 2e_2, \quad e_3 \circ e_1 = 2e_3, \quad e_3 \circ e_3 = e_2
	\end{cases}
	$\\

	$\mathrm{NP}_{63}$&$:$&
	$
	\begin{cases}
		e_1 \cdot e_1 = e_2; \\
		e_1 \circ e_1 = e_1, \quad e_1 \circ e_2 = -e_2, \quad e_2 \circ e_1 = e_2, \quad e_3 \circ e_1 = e_3, \quad e_3 \circ e_3 = e_2
	\end{cases}
	$\\

	$\mathrm{NP}_{64}^{\alpha,\beta}$&$:$&
	$
	\begin{cases}
		e_1 \cdot e_1 = e_2; \\
		e_1 \circ e_1 = \alpha e_2, \quad e_1 \circ e_3 = \beta e_2, \quad e_3 \circ e_3 = e_2
	\end{cases}
	$\\

	$\mathrm{NP}_{65}^{\alpha}$&$:$&
	$
	\begin{cases}
		e_1 \cdot e_1 = e_2; \\
		e_1 \circ e_1 = \alpha e_1, \quad e_1 \circ e_2 = e_2, \quad e_2 \circ e_1 = \alpha e_2, \quad e_3 \circ e_3 = e_3
	\end{cases}
	$\\

	$\mathrm{NP}_{66}^{\alpha}$&$:$&
	$
	\begin{cases}
		e_1 \cdot e_1 = e_2; \\
		e_1 \circ e_1 = e_1+\alpha e_2, \quad e_2 \circ e_1 = e_2, \quad e_3 \circ e_3 = e_3
	\end{cases}
	$\\

	$\mathrm{NP}_{67}^{\alpha}$&$:$&
	$
	\begin{cases}
		e_1 \cdot e_1 = e_2; \\
		e_1 \circ e_1 = \alpha e_2, \quad e_3 \circ e_3 = e_3
	\end{cases}
	$\\

	$\mathrm{NP}_{68}^{\alpha}$&$:$&
	$
	\begin{cases}
		e_1 \cdot e_1 = e_2; \\
		e_1 \circ e_1 = e_1+\alpha e_3, \quad e_1 \circ e_2 = -e_2, \quad e_2 \circ e_1 = e_2, \quad e_3 \circ e_1 = e_2+e_3
	\end{cases}
	$\\

	$\mathrm{NP}_{69}^{\alpha}$&$:$&
	$
	\begin{cases}
		e_1 \cdot e_1 = e_2; \\
		e_1 \circ e_1 = e_3, \quad e_1 \circ e_3 = \alpha e_2, \quad e_3 \circ e_1 = e_2
	\end{cases}
	$\\

	$\mathrm{NP}_{70}^{\alpha} $&$:$&
	$
	\begin{cases}
		e_1 \cdot e_1 = e_2; \\
		e_1 \circ e_3 = \alpha e_2, \quad e_3 \circ e_1 = e_2
	\end{cases}
	$\\

	$\mathrm{NP}_{71}^{\alpha}$&$:$&
	$
	\begin{cases}
		e_1 \cdot e_1 = e_2 ;\\
		e_1 \circ e_1 = \alpha e_2, \quad e_1 \circ e_3 = -e_2, \quad e_3 \circ e_1 = e_2
	\end{cases}
	$\\

	$\mathrm{NP}_{72}^{\alpha} $&$:$&
	$
	\begin{cases}
		e_1 \cdot e_1 = e_2 ;\\
		e_1 \circ e_1 = e_1, \quad e_1 \circ e_2 = \alpha e_2, \quad e_2 \circ e_1 = e_2
	\end{cases}
	$\\

	$\mathrm{NP}_{73}^{\alpha}$&$:$&
	$
	\begin{cases}
		e_1 \cdot e_1 = e_2; \\
		e_1 \circ e_1 = e_1+\alpha e_2, \quad e_2 \circ e_1 = e_2
	\end{cases}
	$
\\
\end{longtable}
\normalsize

\end{theorem}

\subsection{\normalsize The geometric classification of $3$-dimensional Novikov--Poisson algebras}

\begin{theoremA}\label{theoA}
The variety of $3$-dimensional Novikov--Poisson algebras has dimension  $12$ and it has  $4$  irreducible components defined by
\begin{center}
        $\mathcal{C}_{1}=\overline{{\mathcal O}( {\rm NP}_{11}^{\alpha,\beta,\gamma})},$  \ 
        $\mathcal{C}_{2}=\overline{{\mathcal O}( {\rm NP}_{20}^{\alpha,\beta,\gamma})},$ \ 
        $\mathcal{C}_{3}=\overline{{\mathcal O}( {\rm NP}_{25}^{\alpha,\beta,\gamma})},$ \ and
        $\mathcal{C}_{4}=\overline{{\mathcal O}( {\rm NP}_{26}^{\alpha,\beta,\gamma})}.$ \ 
 \end{center}
In particular, there are no rigid algebras.
\end{theoremA}
\begin{proof}
The variety of  $3$-dimensional associative commutative algebras has only one irreducible component $\rm{NP}_{25}^{0,0,0}$  \cite{MS}, 
hence algebras with trivial Novikov multiplication cannot give irreducible components and we will omit them below.

The dimensions of the closures of the families defining the components are as follows
\begin{longtable}{lllllllllllllllllllllllll}

$\dim \mathcal{C}_1$ & $\!\!=\!\!$ & $\dim \mathcal{C}_2$ & $\!\!=\!\!$ & $\dim \mathcal{C}_3$ & $\!\!=\!\!$ & $\dim \mathcal{C}_4$ & $\!\!=\!\!$ & $12.$ \\
\end{longtable}

 The following list of degenerations proves the theorem. 


\begingroup
\setlength{\tabcolsep}{2pt}

\begin{longtable}{lcl}
$\mathrm{NP}_{25}^{\frac{1}{t}, \frac{1}{t}, \frac{1}{t}}$
  & $\xrightarrow{\big(te_1,\ te_2,\ te_3\big)}$
  & $\mathrm{N}_{01}$ \\[0.5ex]
\hline
$\mathrm{NP}_{20}^{\frac{1}{t^2}, \frac{1}{t^2}, \frac{1}{t^2}+\frac{1}{t^3}}$
  & $\xrightarrow{\big(t^2(e_3-e_1+2e_2),\ te_1,\ e_2\big)}$
  & $\mathrm{N}_{02}$ \\[0.5ex]
\hline
$\mathrm{NP}_{26}^{-\frac{1}{t},0, \frac{1}{t}}$
  & $\xrightarrow{\big(te_1,\ t^2e_3,\ te_2\big)}$
  & $\mathrm{N}_{03}$ \\[0.5ex]
\hline
$\mathrm{NP}_{11}^{-\frac{1}{t},0, 0}$
  & $\xrightarrow{\big(te_3,\ t^3e_2,\ t^2e_1\big)}$
  & $\mathrm{N}_{04}$ \\[0.5ex]
\hline
$\mathrm{NP}_{20}^{\frac{1}{2t},-\frac{1}{2t},1}$
  & $\xrightarrow{\big(te_3,\ t^3e_2,\ t^2e_1\big)}$
  & $\mathrm{N}_{05}$ \\[0.5ex]
\hline
$\mathrm{NP}_{20}^{0,-\frac{1}{t},-\frac{1}{4t^2}}$
  & $\xrightarrow{\big(te_3,\ -\frac{t}{2}e_2,\ te_1+\frac{1}{2}e_2\big)}$
  & $\mathrm{N}_{06}$ \\[0.5ex]
\hline
$\mathrm{NP}_{26}^{\frac{\alpha}{t},\frac{\alpha+1}{t},\frac{1}{t}}$
  & $\xrightarrow{\big(te_1,\ e_3,\ te_2\big)}$
  & $\mathrm{N}_{07}^\alpha$ \\[0.5ex]
\hline
$\mathrm{NP}_{11}^{\frac{\alpha}{t},-\frac{(\alpha+1)^2}{t^2},\frac{2(\alpha+1)}{t}}$
  & $\xrightarrow{\big(te_3,\ -(\alpha+1)e_1+te_2,\ e_1\big)}$
  & $\mathrm{N}_{08}^\alpha$ \\[0.5ex]
\hline
$\mathrm{NP}_{11}^{-\frac{1}{t},0,\frac{\alpha-1}{t}}$
  & $\xrightarrow{\big(te_3,\ t^2e_1,\ te_2\big)}$
  & $\mathrm{N}_{09}^\alpha$ \\[0.5ex]
\hline
$\mathrm{NP}_{20}^{\frac{\alpha+1}{t},\frac{\alpha}{t},0}$
  & $\xrightarrow{\big(te_3,\ te_2,\ \frac{t}{\sqrt{\alpha+\frac{1}{2}}}e_1\big)}$
  & $\mathrm{N}_{10}^\alpha$ \\[0.5ex]
\hline
$\mathrm{NP}_{11}^{\frac{\beta}{t},-\frac{(\beta+1)(\alpha+\beta+t)}{t^2},\frac{\alpha+2\beta+t+1}{t}}$
  & $\xrightarrow{\big(te_3,-\frac{\alpha+\beta+t}{t}e_1+e_2,\ -\frac{\beta+1}{t}e_1+e_2\big)}$
  & $\mathrm{N}_{11}^{\alpha,\beta}$ \\[0.5ex]
\hline

\multicolumn{2}{l}{${\rm NP}_{20}^{\frac{\alpha t - \beta t - 2}{2 t},\ \frac{\alpha t - \beta t + 2}{2 t},\ -\frac{(\alpha t - \beta t - 2)(\alpha t + \beta t + 2)}{t^2(\alpha+\beta)^2}} \quad\quad\quad \quad \quad\xrightarrow{ \quad \quad\quad\quad\quad\quad\quad\quad  } $}&$ {\rm NP}_{01}^{\alpha,\beta}$\\
\multicolumn{3}{l}{$E_1^t= \frac{2 t}{\alpha + \beta}e_1 - \frac{4(\alpha t - \beta t - 2)}{(\alpha+\beta)^3}e_2 + t e_3$ \quad
$E_2^t= -\frac{8 t^2}{(\alpha+\beta)^2}e_2$ \quad
$E_3^t= -\frac{2 t}{\alpha + \beta}e_1 - \frac{4(\alpha t + 3\beta t + 2)}{(\alpha+\beta)^3}e_2 + t e_3$} \\[0.5ex]
\hline

\multicolumn{2}{l}{${\rm NP}_{20}^{-\frac{2\alpha t^3 - t + 1}{t^3},\ \frac{t+1}{t^3},\ \frac{(2\alpha t^3 - t + 1)(3\alpha t^3 - t + 1)}{t^2 (2\alpha t^2 - 1)^2}}
\quad \quad\quad\quad\quad\xrightarrow{ \quad\quad\quad\quad\quad\quad\quad\quad\quad   }$}&$  {\rm NP}_{02}^{\alpha}$\\

\multicolumn{3}{l}{$E_1^t= \frac{t^5}{2\alpha t^2 - 1}e_1 + \frac{t^6(2\alpha t^3 - t + 1)}{(2\alpha t^2 - 1)^3}e_2 + t^3 e_3$ \quad
$E_2^t= \frac{3 t^{10}}{2\alpha t^2 - 1}e_1 + \frac{3 t^{11}(4\alpha t^3 - 2 t + 1)}{(2\alpha t^2 - 1)^3}e_2 + t^8 e_3$} \\

\multicolumn{3}{l}{$E_3^t= \frac{2 t^7}{2\alpha t^2 - 1}e_1 + \frac{t^8(6\alpha t^3 - 3 t + 2)}{(2\alpha t^2 - 1)^3}e_2 + t^5 e_3$} \\[0.5ex]
\hline

${\rm NP}_{20}^{- \frac{2 \alpha t^{3} + 1}{t^{3}},\frac{1}{t^{3}},\frac{\left(2 \alpha t^{3} + 1\right) \left(3 \alpha t^{3} + 1\right)}{4 \alpha^{2} t^{6}}}  $&$\xrightarrow{ \quad\quad\quad\quad\quad\quad\quad\quad  }$&$  {\rm NP}_{03}^{\alpha}$\\

\multicolumn{3}{l}{$E_1^t= \frac{t^{3}}{2 \alpha}e_1 + \frac{2 \alpha t^{3} + 1}{8 \alpha^{3}}e_2 + t^{3}e_3$ \quad
$E_2^t= \frac{3 t^{8}}{2 \alpha}e_1 + \frac{3 t^{5} \left(4 \alpha t^{3} + 1\right)}{8 \alpha^{3}}e_2 + t^{8}e_3$ }\\[0.5ex]
\hline

\multicolumn{2}{l}{
${\rm NP}_{20}^{\frac{2\alpha t^2 - 2\beta t^2 + \gamma}{t^2},\ \frac{\gamma}{t^2},\ \frac{2\alpha^2 t^4 - 7\alpha\beta t^4 + 3\alpha\gamma t^2 - 2\alpha t^6 + 6\beta^2 t^4 - 5\beta\gamma t^2 + 4\beta t^6 + \gamma^2 - 2\gamma t^4 + t^8}{(-\alpha t^2 + 2\beta t^2 - \gamma + t^4)^2}}
\ \xrightarrow{ \quad\quad\quad\quad\quad}$}&$ {\rm NP}_{04}^{\alpha,\beta,\gamma}$\\

\multicolumn{3}{l}{$E_1^t= \frac{t^5}{-\alpha t^2 + 2\beta t^2 - \gamma + t^4}e_1 - \frac{t^7(\alpha t^2 - 2\beta t^2 + \gamma)}{(-\alpha t^2 + 2\beta t^2 - \gamma + t^4)^3}e_2 + t^3 e_3$ }\\
\multicolumn{3}{l}{$E_2^t= \frac{3 t^{10}}{-\alpha t^2 + 2\beta t^2 - \gamma + t^4}e_1 + \frac{3 t^{12}(-2\alpha t^2 + 4\beta t^2 - 2\gamma + t^4)}{(-\alpha t^2 + 2\beta t^2 - \gamma + t^4)^3}e_2 + t^8 e_3$} \\
\multicolumn{3}{l}{$E_3^t= \frac{2 t^7}{-\alpha t^2 + 2\beta t^2 - \gamma + t^4}e_1 + \frac{t^9(-3\alpha t^2 + 6\beta t^2 - 3\gamma + t^4)}{(-\alpha t^2 + 2\beta t^2 - \gamma + t^4)^3}e_2 + t^5 e_3$} \\[0.5ex]
\hline

$\mathrm{NP}_{20}^{\alpha, -\alpha + 2\beta, 0}$
  & $\xrightarrow{\big(t^2 e_1 + t^3 e_3,\ 3 t^7 e_1 + 3 t^6 e_2 + t^8 e_3,\ 2 t^4 e_1 + t^3 e_2 + t^5 e_3\big)}$
  & $\mathrm{NP}_{05}^{\alpha,\beta}$ \\[0.5ex]
\hline

${\rm NP}_{20}^{\frac{-\alpha + 2\beta t + 2}{t},\ \frac{\alpha}{t},\ \frac{(-\alpha + \beta t + 2)(-\alpha + 2\beta t + 2)}{(\beta t + 1)^2}}
$&$\xrightarrow{ \quad\quad\quad\quad\quad\quad\quad   } $&$ {\rm NP}_{06}^{\alpha,\beta}$\\
\multicolumn{3}{l}{$E_1^t= -\frac{t^2}{\beta t + 1}e_1 + \frac{t^3(-\alpha + \beta t + 2)}{(\beta t + 1)^3}e_2 + t e_3$ \quad
$E_2^t= -\frac{2 t^3}{\beta t + 1}e_1 + \frac{t^4(-2\alpha + 3\beta t + 5)}{(\beta t + 1)^3}e_2 + t^2 e_3$} \\
\multicolumn{3}{l}{$E_3^t= -\frac{3 t^4}{\beta t + 1}e_1 + \frac{3 t^5(-\alpha + 2\beta t + 3)}{(\beta t + 1)^3}e_2 + t^3 e_3$} \\[0.5ex]
\hline

${\rm NP}_{20}^{2\beta,\ \frac{2}{t},\ \frac{t^2(\alpha^2 t^2 - 2\alpha\beta t - \alpha + 2\beta^2)}{(\alpha t^2 - \beta t - 1)^2}}$&$  \xrightarrow{ \quad\quad\quad\quad\quad\quad\quad  }  $&${\rm NP}_{07}^{\alpha,\beta}$\\
\multicolumn{3}{l}{$E_1^t= \frac{t^2}{\alpha t^2 - \beta t - 1}e_1 - \frac{\beta t^4}{(\alpha t^2 - \beta t - 1)^3}e_2 + t e_3$ \quad
$E_2^t= \frac{2 t^3}{\alpha t^2 - \beta t - 1}e_1 + \frac{t^4(\alpha t^2 - 3\beta t - 1)}{(\alpha t^2 - \beta t - 1)^3}e_2 + t^2 e_3$} \\
\multicolumn{3}{l}{$E_3^t= \frac{3 t^4}{\alpha t^2 - \beta t - 1}e_1 + \frac{3 t^5(\alpha t^2 - 2\beta t - 1)}{(\alpha t^2 - \beta t - 1)^3}e_2 + t^3 e_3$ }\\[0.5ex]
\hline

\multicolumn{2}{l}{
$\mathrm{NP}_{20}^{-\frac{2\alpha t + 1}{t},\ \frac{1}{t},\ \frac{(2\alpha t + 1)(3\alpha t + 1)}{4\alpha^2 t^2}}\ \xrightarrow{\big(\frac{t}{2\alpha}e_1 + \frac{2\alpha t + 1}{8\alpha^3}e_2 + t e_3,\ \frac{t^2}{\alpha}e_1 + \frac{t(3\alpha t + 1)}{4\alpha^3}e_2 + t^2 e_3,\ \frac{3 t^3}{2\alpha}e_1 + \frac{3 t^2(4\alpha t + 1)}{8\alpha^3}e_2 + t^3 e_3\big)}$}
  & $\mathrm{NP}_{08}^{\alpha}$ \\[0.5ex]
\hline
$\mathrm{NP}_{20}^{\beta, 2\alpha - \beta, 0}$
  & $\xrightarrow{\big(e_1 + t e_3,\ 2 t e_1 + e_2 + t^2 e_3,\ 3 t^2 e_1 + 3 t e_2 + t^3 e_3\big)}$
  & $\mathrm{NP}_{09}^{\alpha,\beta}$ \\[0.5ex]
\hline
$\mathrm{NP}_{20}^{\alpha, \alpha, 1}$
  & $\xrightarrow{\big(e_1 + t e_3,\ 2 t e_1 + e_2 + t^2 e_3,\ 3 t^2 e_1 + 3 t e_2 + t^3 e_3\big)}$
  & $\mathrm{NP}_{10}^{\alpha}$ \\[0.5ex]
\hline
$\mathrm{NP}_{11}^{\alpha,\ -\beta\gamma + t,\ \beta + \gamma}$
  & $\xrightarrow{\big(e_1,\ -\frac{\beta}{t}e_1 + \frac{1}{t}e_2,\ e_3\big)}$
  & $\mathrm{NP}_{12}^{\alpha,\beta,\gamma}$ \\[0.5ex]
\hline
$\mathrm{NP}_{11}^{\alpha,\ -\beta^2,\ 2\beta}$
  & $\xrightarrow{\big(e_1,\ -\frac{\beta}{t}e_1 + \frac{1}{t}e_2,\ e_3\big)}$
  & $\mathrm{NP}_{13}^{\alpha,\beta}$ \\[0.5ex]
\hline
$\mathrm{NP}_{11}^{\alpha,\ \beta,\ \gamma}$
  & $\xrightarrow{\big(\frac{\beta}{t}e_1,\ \frac{1}{t}e_2,\ e_3\big)}$
  & $\mathrm{NP}_{14}^{\alpha,\beta,\gamma}$ \\[0.5ex]
\hline
$\mathrm{NP}_{11}^{\alpha,\ -\beta^2,\ 2\beta}$
  & $\xrightarrow{\big(\frac{1}{t}e_1,\ -\frac{\beta}{t^2}e_1 + \frac{1}{t^2}e_2,\ e_3\big)}$
  & $\mathrm{NP}_{15}^{\alpha,\beta}$ \\[0.5ex]
\hline
$\mathrm{NP}_{26}^{\alpha,\ \beta,\ -\frac{1}{t}}$
  & $\xrightarrow{\big(t e_2 + e_3,\ e_3,\ e_1\big)}$
  & $\mathrm{NP}_{16}^{\alpha,\beta}$ \\[0.5ex]
\hline
$\mathrm{NP}_{26}^{\alpha,\ \beta,\ 0}$
  & $\xrightarrow{\big(t e_2 + e_3,\ e_3,\ e_1\big)}$
  & $\mathrm{NP}_{17}^{\alpha,\beta}$ \\[0.5ex]
\hline
$\mathrm{NP}_{26}^{\alpha,\ \beta,\ -\frac{1}{t}}$
  & $\xrightarrow{\big(t e_2 + \frac{1}{t}e_3,\ \frac{1}{t}e_3,\ e_1\big)}$
  & $\mathrm{NP}_{18}^{\alpha,\beta}$ \\[0.5ex]
\hline
$\mathrm{NP}_{26}^{\alpha,\ \beta,\ 0}$
  & $\xrightarrow{\big(t e_2 + \frac{1}{t}e_3,\ \frac{1}{t}e_3,\ e_1\big)}$
  & $\mathrm{NP}_{19}^{\alpha,\beta}$ \\[0.5ex]
\hline
$\mathrm{NP}_{20}^{\alpha,\ \beta,\ \frac{\alpha - \beta}{2 t^2}}$
  & $\xrightarrow{\big(\frac{1}{t}e_1 - \frac{1}{t^3}e_2,\ \frac{1}{t^2}e_2,\ e_3\big)}$
  & $\mathrm{NP}_{21}^{\alpha,\beta}$ \\[0.5ex]
\hline
$\mathrm{NP}_{20}^{\alpha,\ \alpha,\ \frac{\beta}{t}}$
  & $\xrightarrow{\big(\frac{1}{t}e_1 - \frac{1}{t^3}e_2,\ \frac{1}{t^2}e_2,\ e_3\big)}$
  & $\mathrm{NP}_{22}^{\alpha,\beta}$ \\[0.5ex]
\hline
$\mathrm{NP}_{20}^{\beta,\ \alpha,\ 0}$
  & $\xrightarrow{\big(\frac{1}{t}e_1,\ \frac{1}{t^2}e_2,\ e_3\big)}$
  & $\mathrm{NP}_{23}^{\alpha,\beta}$ \\[0.5ex]
\hline
$\mathrm{NP}_{20}^{\alpha,\ \alpha,\ \frac{1}{t}}$
  & $\xrightarrow{\big(\frac{1}{t}e_1,\ \frac{1}{t^2}e_2,\ e_3\big)}$
  & $\mathrm{NP}_{24}^{\alpha}$ \\[0.5ex]
\hline
$\mathrm{NP}_{26}^{\alpha,\beta,\gamma}$
  & $\xrightarrow{\big(e_1,\ e_2,\ \frac{1}{t} e_3\big)}$
  & $\mathrm{NP}_{27}^{\alpha,\beta,\gamma}$ \\[0.5ex]
\hline
$\mathrm{NP}_{25}^{\alpha, \beta, \frac{1}{t}}$
  & $\xrightarrow{\big(e_1,\ e_2,\ te_3\big)}$
  & $\mathrm{NP}_{28}^{\alpha,\beta}$ \\[0.5ex]
\hline
$\mathrm{NP}_{25}^{\alpha, \beta, 0}$
  & $\xrightarrow{\big(e_1,\ e_2,\ te_3\big)}$
  & $\mathrm{NP}_{29}^{\alpha,\beta}$ \\[0.5ex]
\hline
$\mathrm{NP}_{26}^{\frac{\beta}{t},\ \frac{1}{t},\ \alpha}$
  & $\xrightarrow{\big(e_2,\ t e_1 - t^2 e_3,\ t^2 e_1 - 2 t^3 e_3\big)}$
  & $\mathrm{NP}_{30}^{\alpha,\beta}$ \\[0.5ex]
\hline
$\mathrm{NP}_{26}^{-\frac{1}{2t},\ \beta - 2,\ \alpha}$
  & $\xrightarrow{\big(e_2,\ -2 t e_1 + t e_2 - t e_3,\ 2 t^2 e_1 - t^2 e_2 + 3 t^2 e_3\big)}$
  & $\mathrm{NP}_{31}^{\alpha,\beta}$ \\[0.5ex]
\hline
$\mathrm{NP}_{25}^{\alpha, 2\beta, 0}$
  & $\xrightarrow{\big(e_1,\ te_2 - te_3,\ t^2e_2 + t^2e_3\big)}$
  & $\mathrm{NP}_{32}^{\alpha,\beta}$ \\[0.5ex]
\hline
$\mathrm{NP}_{26}^{-\frac{\beta}{3t},\ -\frac{1}{3t},\ \alpha}$
  & $\xrightarrow{\big(e_2,\ -3 t e_1 - 4 t e_2,\ 2 t e_2 + 5 t e_3\big)}$
  & $\mathrm{NP}_{33}^{\alpha,\beta}$ \\[0.5ex]
\hline
$\mathrm{NP}_{26}^{\frac{1}{t^2},\ \frac{2}{t},\ \alpha}$
  & $\xrightarrow{\big(e_2,\ t^2 e_1 - 2 t^2 e_2 - 4 t^2 e_3,\ -2 t^3 e_1 - 4 t^3 e_2\big)}$
  & $\mathrm{NP}_{34}^{\alpha}$ \\[0.5ex]
\hline
$\mathrm{NP}_{26}^{\frac{1}{t^2},\ 0,\ \alpha}$
  & $\xrightarrow{\big(e_2,\ t^2 e_1 - 2 t^2 e_2 - 4 t^2 e_3,\ -2 t^3 e_1 - 4 t^3 e_2\big)}$
  & $\mathrm{NP}_{35}^{\alpha}$ \\[0.5ex]
\hline
$\mathrm{NP}_{25}^{\alpha, 1, -1}$
  & $\xrightarrow{\big(e_1,\ t e_2 + t e_3,\ t^2 e_2 - t^2 e_3\big)}$
  & $\mathrm{NP}_{36}^{\alpha}$ \\[0.5ex]
\hline
$\mathrm{NP}_{25}^{\alpha, 0, 0}$
  & $\xrightarrow{\big(e_1,\ t e_2,\ t e_3\big)}$
  & $\mathrm{NP}_{37}^{\alpha}$ \\[0.5ex]
\hline
$\mathrm{NP}_{25}^{\alpha, \frac{1}{t}, \frac{1}{t}}$
  & $\xrightarrow{\big(e_1,\ t e_2,\ t e_3\big)}$
  & $\mathrm{NP}_{38}^{\alpha}$ \\[0.5ex]
\hline
$\mathrm{NP}_{25}^{\alpha, \frac{1}{t}, 0}$
  & $\xrightarrow{\big(e_1,\ t e_2,\ t e_3\big)}$
  & $\mathrm{NP}_{39}^{\alpha}$ \\[0.5ex]
\hline

${\rm NP}_{11}^{\frac{1}{t},\ -\frac{\alpha\beta}{t^2},\ \frac{\alpha+\beta}{t}}  $&$\xrightarrow{ \quad\quad\quad\quad\quad\quad\quad  } $&$ {\rm NP}_{40}^{\alpha,\beta}$\\
\multicolumn{3}{l}{$E_1^t= -\frac{t^2(\alpha+2\beta-2)}{2\alpha(\beta-1)}e_1 + \frac{t^3}{2\alpha(\beta-1)}e_2 + t e_3$ \quad
$E_2^t= -\frac{t^3(\alpha+2\beta-2)}{\alpha(\beta-1)}e_1 + \frac{t^4}{\alpha(\beta-1)}e_2 + t^2 e_3$} \\
\multicolumn{3}{l}{$E_3^t= \frac{1}{2 t^8(\beta-1)}e_1 - \frac{1}{2\alpha t^7(\beta-1)}e_2$ }\\[0.5ex]
\hline
${\rm NP}_{11}^{\frac{1}{t},\ -\frac{\alpha^2 + t^{13}}{t^2},\ \frac{2\alpha + t^{13}}{t}} $&$ \xrightarrow{  \quad\quad\quad\quad\quad\quad\quad   } $&$ {\rm NP}_{41}^{\alpha}$\\
\multicolumn{3}{l}{$E_1^t= -\frac{t^2}{\alpha}e_1 + t e_3$ \quad
$E_2^t= -\frac{2 t^3}{\alpha}e_1 + t^2 e_3$ \quad
$E_3^t= \frac{\alpha + t^{13}}{\alpha t^{10}(\alpha-1)}e_1 - \frac{1}{\alpha t^9(\alpha-1)}e_2$} \\[0.5ex]
\hline

$\mathrm{NP}_{11}^{\frac{1}{t},\ 0,\ \frac{\beta}{t}}$
  & $\xrightarrow{\big(\frac{t}{\alpha}e_1 + t e_3,\ \frac{2 t^2}{\alpha}e_1 + t^2 e_3,\ \frac{1}{\alpha t^{17}(\beta-1)}e_2\big)}$
  & $\mathrm{NP}_{42}^{\alpha,\beta}$ \\[0.5ex]
\hline
\multicolumn{2}{l}{
$\mathrm{NP}_{11}^{\frac{1}{t},\ \frac{\beta}{t^2},\ \frac{\alpha}{t}} \quad \quad  \quad \quad \quad \quad  \xrightarrow{\big(\frac{\alpha t^2}{\beta}e_1 - \frac{t^3}{\beta}e_2 + t e_3,\ \frac{2\alpha t^3}{\beta}e_1 - \frac{2 t^4}{\beta}e_2 + t^2 e_3,\ -\frac{t^3(\alpha^2+\beta)}{\beta}e_1 + \frac{\alpha t^4}{\beta}e_2\big)}$}
  & $\mathrm{NP}_{43}^{\alpha,\beta}$ \\[0.5ex]
\hline
\multicolumn{2}{l}{
$\mathrm{NP}_{11}^{\frac{1}{t},\ 0,\ \frac{\beta}{t}}
\quad \quad\quad \quad \quad \  \quad \xrightarrow{\big(\frac{t}{\alpha}e_1 + \frac{t^2}{\alpha(\alpha t - \beta)}e_2 + t e_3,\ \frac{2 t^2}{\alpha}e_1 + \frac{2 t^3}{\alpha(\alpha t - \beta)}e_2 + t^2 e_3,\ -\frac{\beta t^2}{\alpha}e_1 + \frac{t^3}{\alpha}e_2\big)}$}
  & $\mathrm{NP}_{44}^{\alpha,\beta}$ \\[0.5ex]
\hline
$\mathrm{NP}_{11}^{\frac{1}{t},\ 0,\ \frac{\alpha}{t}}$
  & $\xrightarrow{\big(-\frac{t^3}{\alpha^2}e_2 + t e_3,\ -\frac{2 t^4}{\alpha^2}e_2 + t^2 e_3,\ t^2 e_1 - \frac{t^3}{\alpha}e_2\big)}$
  & $\mathrm{NP}_{45}^{\alpha}$ \\[0.5ex]
\hline
$\mathrm{NP}_{11}^{\frac{1}{t},\ \frac{\alpha}{t},\ 0}$
  & $\xrightarrow{\big(e_1 + t e_3,\ 2 t e_1 + t^2 e_3,\ t e_2\big)}$
  & $\mathrm{NP}_{46}^{\alpha}$ \\[0.5ex]
\hline
$\mathrm{NP}_{11}^{0,\ 0,\ \frac{1}{t}}$
  & $\xrightarrow{\big(\frac{t}{\alpha}e_1 + \frac{t^2}{\alpha(\alpha t - 1)}e_2 + t e_3,\ \frac{2 t^2}{\alpha}e_1 + \frac{2 t^3}{\alpha(\alpha t - 1)}e_2 + t^2 e_3,\ \frac{t^4}{\alpha t - 1}e_2\big)}$
  & $\mathrm{NP}_{47}^{\alpha}$ \\[0.5ex]
\hline
$\mathrm{NP}_{11}^{0,\ -\frac{\alpha}{t^2},\ \frac{\alpha+1}{t}}$
  & $\xrightarrow{\big(-t^2 e_1 + t e_3,\ -2 t^3 e_1 + t^2 e_3,\ \frac{1}{\alpha t^{17}}e_1 - \frac{1}{\alpha t^{16}}e_2\big)}$
  & $\mathrm{NP}_{48}^{\alpha}$ \\[0.5ex]
\hline
$\mathrm{NP}_{11}^{0,\ -\frac{1}{t^2},\ \frac{t^{20}+2}{t}}$
  & $\xrightarrow{\big(-t^2 e_1 + t e_3,\ -2 t^3 e_1 + t^2 e_3,\ \frac{t^{20}+1}{t^{17}}e_1 - \frac{1}{t^{16}}e_2\big)}$
  & $\mathrm{NP}_{49}$ \\[0.5ex]
\hline
$\mathrm{NP}_{26}^{2\alpha,\ -\frac{1}{t},\ 0}$
  & $\xrightarrow{\big(-t e_1 + t e_2 - t e_3,\ 2 t^2 e_1 - 2 t^2 e_3,\ -e_3\big)}$
  & $\mathrm{NP}_{50}^{\alpha}$ \\[0.5ex]
\hline
\multicolumn{2}{l}{
$\mathrm{NP}_{11}^{0,\ \frac{1}{t^2},\ \frac{\alpha}{t}}
\quad \quad \quad \quad \xrightarrow{\big(\alpha t^2 e_1 - t^3 e_2 + t e_3,\ 2\alpha t^3 e_1 - 2 t^4 e_2 + t^2 e_3,\ -t^3(\alpha^2+1)e_1 + \alpha t^4 e_2\big)}$}
  & $\mathrm{NP}_{51}^{\alpha}$ \\[0.5ex]
\hline
$\mathrm{NP}_{11}^{0,\ \frac{1}{t},\ 0}$
  & $\xrightarrow{\big(e_1 + t e_3,\ 2 t e_1 + t^2 e_3,\ t e_2\big)}$
  & $\mathrm{NP}_{52}$ \\[0.5ex]
\hline
$\mathrm{NP}_{11}^{0,\ \frac{t+3}{3},\ 0}$
  & $\xrightarrow{\big(-3 t e_1 - 4 t e_2 + t e_3,\ -2 t^2 e_1 - 3 t^2 e_2,\ 3 e_1 + 3 e_2\big)}$
  & $\mathrm{NP}_{53}$ \\[0.5ex]
\hline
$\mathrm{NP}_{11}^{0,\ \alpha - t,\ \alpha - 1}$
  & $\xrightarrow{\big(e_1 + e_2 + t e_3,\ 2 t e_1 + 2 t e_2 + t^2 e_3,\ t^3 e_2\big)}$
  & $\mathrm{NP}_{54}^{\alpha}$ \\[0.5ex]
\hline
$\mathrm{NP}_{26}^{0,\ -\frac{1}{t^3},\ -\frac{1}{t^3}}$
  & $\xrightarrow{\big(-t^3 e_1 + t^4 e_2 - t^6 e_3,\ -t^6 e_1 - t^6 e_2,\ -t^5 e_2\big)}$
  & $\mathrm{NP}_{55}$ \\[0.5ex]
\hline
$\mathrm{NP}_{25}^{-\frac{1}{2}, 0, 0}$
  & $\xrightarrow{\big(-2 t e_1 - 2 t e_2 - t e_3,\ 4 t^2 e_1 + 4 t^2 e_2 + t^2 e_3,\ -2 t^2 e_1\big)}$
  & $\mathrm{NP}_{56}$ \\[0.5ex]
\hline
$\mathrm{NP}_{25}^{\alpha, \alpha, 0}$
  & $\xrightarrow{\big(t e_1 - t e_2,\ t^2 e_1 + t^2 e_2,\ t e_3\big)}$
  & $\mathrm{NP}_{57}^{\alpha}$ \\[0.5ex]
\hline
$\mathrm{NP}_{20}^{\frac{2\alpha t - 1}{t},\ \frac{1}{t},\ 0}$
  & $\xrightarrow{\big(t e_1,\ t^2 e_2,\ t e_3\big)}$
  & $\mathrm{NP}_{58}^{\alpha}$ \\[0.5ex]
\hline
$\mathrm{NP}_{20}^{\frac{2\alpha t^2 - 1}{t^2},\ \frac{1}{t^2},\ 0}$
  & $\xrightarrow{\big(t e_1,\ t^2 e_2,\ t^2 e_3\big)}$
  & $\mathrm{NP}_{59}^{\alpha}$ \\[0.5ex]
\hline
$\mathrm{NP}_{20}^{-\frac{\alpha-2}{t^2},\ \frac{\alpha}{t^2},\ 2-\alpha}$
  & $\xrightarrow{\big(-t^4 e_1 + t^6 e_2 + t^2 e_3,\ -2 t^6 e_1 + 3 t^8 e_2 + t^4 e_3,\ t^3 e_3\big)}$
  & $\mathrm{NP}_{60}^{\alpha}$ \\[0.5ex]
\hline
$\mathrm{NP}_{20}^{\frac{1}{t^2},\ \frac{1}{t^2},\ \alpha t + 1}$
  & $\xrightarrow{\big(-t^4 e_1 + t^6(\alpha t + 1)e_2 + t^2 e_3,\ -2 t^6 e_1 + t^8(2\alpha t + 3)e_2 + t^4 e_3,\ t^3 e_3\big)}$
  & $\mathrm{NP}_{61}^{\alpha}$ \\[0.5ex]
\hline
$\mathrm{NP}_{20}^{0,\ \frac{2}{t^2},\ \frac{\alpha t^2(\alpha t^2 - 4)}{(\alpha t^2 - 2)^2}}$
  & $\xrightarrow{\big(\frac{2 t^4}{\alpha t^2 - 2}e_1 - \frac{8 t^6}{(\alpha t^2 - 2)^3}e_2 + t^2 e_3,\ \frac{4 t^6}{\alpha t^2 - 2}e_1 + \frac{4 t^8(\alpha t^2 - 6)}{(\alpha t^2 - 2)^3}e_2 + t^4 e_3,\ t^3 e_3\big)}$
  & $\mathrm{NP}_{62}^{\alpha}$ \\[0.5ex]
\hline
$\mathrm{NP}_{20}^{\frac{2t-1}{t^5},\ \frac{1}{t^5},\ \frac{2t-1}{t}}$
  & $\xrightarrow{\big(-t^9 e_1 + t^{13} e_2 + t^5 e_3,\ -2 t^{14} e_1 + 3 t^{18} e_2 + t^{10} e_3,\ t^7 e_3\big)}$
  & $\mathrm{NP}_{63}$ \\[0.5ex]
\hline
$\mathrm{NP}_{20}^{\alpha - \beta + 1,\ \alpha - \beta + 1,\ -\frac{1}{\beta-1}}$
  & $\xrightarrow{\big(\frac{t}{\beta-1}e_1 - \frac{t}{(\beta-1)^3}e_2 + t e_3,\ \frac{2 t^2}{\beta-1}e_1 + \frac{t^2(\beta-3)}{(\beta-1)^3}e_2 + t^2 e_3,\ t e_3\big)}$
  & $\mathrm{NP}_{64}^{\alpha,\beta}$ \\[0.5ex]
\hline
$\mathrm{NP}_{26}^{-\frac{\alpha}{t},\ -\frac{1}{t},\ -\frac{1}{t}}$
  & $\xrightarrow{\big(-t e_1 - t^2 e_3,\ -t^2 e_1 - t^2 e_2,\ -t e_2\big)}$
  & $\mathrm{NP}_{65}^{\alpha}$ \\[0.5ex]
\hline
$\mathrm{NP}_{26}^{-\frac{1}{t},\ \alpha - 1,\ -\frac{1}{t}}$
  & $\xrightarrow{\big(-t e_1 - t e_3,\ -t^2 e_1 - t^2 e_2,\ -t e_2\big)}$
  & $\mathrm{NP}_{66}^{\alpha}$ \\[0.5ex]
\hline
$\mathrm{NP}_{25}^{\alpha, \alpha, \frac{1}{t}}$
  & $\xrightarrow{\big(t e_1 - t e_2,\ t^2 e_1 + t^2 e_2,\ t e_3\big)}$
  & $\mathrm{NP}_{67}^{\alpha}$ \\[0.5ex]
\hline

${\rm NP}_{20}^{\frac{2\alpha t^2 - 1}{\alpha t^{10}},\ \frac{1}{\alpha t^{10}},\ \frac{2\alpha t^2 - 1}{\alpha t^2}} $&$ \xrightarrow{ \quad\quad\quad\quad\quad\quad\quad  }
$&${\rm NP}_{68}^{\alpha}$\\[0.5ex]

\multicolumn{3}{l}{$E_1^t= \frac{\alpha t^{18}(2\alpha t^8 - \alpha t^2 + 1)}{\alpha t^2 - 1}e_1 - \frac{\alpha t^{26}(3\alpha t^8 - \alpha t^2 + 1)}{\alpha t^2 - 1}e_2 - \frac{\alpha t^{10}(\alpha t^8 - \alpha t^5 - \alpha t^2 + t^3 + 1)}{\alpha t^2 - 1}e_3$} \\
\multicolumn{3}{l}{$E_2^t= -2\alpha^2 t^{28}e_1 + 3\alpha^2 t^{36}e_2 + \alpha^2 t^{20}e_3$ \quad
$E_3^t= \frac{2\alpha^2 t^{28}}{\alpha t^2 - 1}e_1 - \frac{3\alpha^2 t^{36}}{\alpha t^2 - 1}e_2 - \frac{\alpha t^{15}(\alpha t^5 - \alpha t^2 + 1)}{\alpha t^2 - 1}e_3$} \\[0.5ex]
\hline

$\mathrm{NP}_{20}^{0,\ 0,\ \alpha}$
  & $\xrightarrow{\big(-t^3 e_1 + t^4(\alpha+1)e_2 + t^2 e_3,\ -2 t^5 e_1 + t^6(2\alpha+3)e_2 + t^4 e_3,\ t^3 e_3\big)}$
  & $\mathrm{NP}_{69}^{\alpha}$ \\[0.5ex]
\hline
$\mathrm{NP}_{26}^{-\frac{\alpha}{t},\ -\frac{1}{t},\ 0}$
  & $\xrightarrow{\big(-t e_2 - t e_3,\ -t^2 e_3,\ -t^2 e_1 - t^2 e_2 - t^2 e_3\big)}$
  & $\mathrm{NP}_{70}^{\alpha}$ \\[0.5ex]
\hline
$\mathrm{NP}_{20}^{2\alpha,\ 0,\ -1}$
  & $\xrightarrow{\big(t e_1,\ t^2 e_2,\ -t e_3\big)}$
  & $\mathrm{NP}_{71}^{\alpha}$ \\[0.5ex]
\hline
$\mathrm{NP}_{26}^{-\frac{1}{t},\ -\frac{\alpha}{t},\ 0}$
  & $\xrightarrow{\big(-t e_1 - \frac{t^2}{\alpha} e_3,\ -t^2 e_1 - t^2 e_2,\ -t e_2\big)}$
  & $\mathrm{NP}_{72}^{\alpha}$ \\[0.5ex]
\hline
$\mathrm{NP}_{26}^{-\frac{1}{t},\ \alpha - 1,\ 0}$
  & $\xrightarrow{\big(-t e_1 - t e_3,\ -t^2 e_1 - t^2 e_2,\ -t e_2\big)}$
  & $\mathrm{NP}_{73}^{\alpha}$ \\[0.5ex]
\hline
\end{longtable}
\endgroup

\end{proof}

\section{\large The geometric classification of Gelfand--Dorfman   algebras}

\begin{definition}[see \cite{KSO19}]
A Gelfand--Dorfman algebra is a triple $({\rm A}, \circ, [\cdot, \cdot])$
equipped with two bilinear operations such that 
$({\rm A}, \circ)$ is a Novikov algebra, 
$({\rm A}, [\cdot, \cdot])$ is a Lie algebra, 
and the following compatibility identity holds 
\begin{center}
$[x, y \circ z] - [z, y \circ x] + [y, x] \circ z -  [y, z] \circ x - y \circ [x, z]\  =\  0.
$\end{center}
\end{definition}
\subsection{\normalsize The algebraic classification of $3$-dimensional Gelfand--Dorfman algebras}

\begin{theorem}[see \cite{ZHSC-GD}]
	Let \((\mathrm{G},\circ,[\cdot,\cdot])\) be a complex $3$-dimensional Gelfand--Dorfman algebra with $[\cdot,\cdot]\neq 0.$  
    Then \((\mathrm{G},\circ,[\cdot,\cdot])\) is isomorphic to one of the following algebras:

	\begingroup\setlength{\tabcolsep}{2pt}
	\begin{longtable}{lcl}

$\mathrm{G}_{01}^{\alpha}$&$:$&$
\begin{cases}
[e_1, e_2] = e_3; \\
e_1 \circ e_2 = e_2 + \alpha e_3, \quad e_1 \circ e_3 = e_3
\end{cases}
$\\

$\mathrm{G}_{02}^{\alpha}$&$:$&$
\begin{cases}
[e_1, e_2] = e_3; \\
e_1 \circ e_1 = e_3, \quad e_1 \circ e_2 = e_1, \quad e_2 \circ e_1 = e_1 + \alpha e_3, \\
e_2 \circ e_2 = e_2, \quad e_2 \circ e_3 = e_3, \quad e_3 \circ e_2 = e_3
\end{cases}
$\\

$\mathrm{G}_{03}^{\alpha,\beta}$&$:$&$
\begin{cases}
[e_1, e_2] = e_3; \\
e_1 \circ e_1 = e_1, \quad e_1 \circ e_2 = \alpha e_2+\beta e_3, \quad e_1 \circ e_3 = \alpha e_3, \quad  e_2 \circ e_1 = e_2, \quad e_3 \circ e_1 = e_3
\end{cases}
$\\

$\mathrm{G}_{04}^{\alpha}$&$:$&$
\begin{cases}
[e_1, e_2] = e_3; \\
e_1 \circ e_1 = e_1 + e_2, \quad e_1 \circ e_2 = \alpha e_3, \quad e_2 \circ e_1 = e_2, \quad e_3 \circ e_1 = e_3
\end{cases}
$\\

$\mathrm{G}_{05}$&$:$&$
\begin{cases}
[e_1, e_2] = e_3; \\
e_1 \circ e_1 = e_1 + e_3, \quad e_2 \circ e_1 = e_2, \quad e_3 \circ e_1 = e_3
\end{cases}
$\\

$\mathrm{G}_{06}^{\alpha}$&$:$&$
\begin{cases}
[e_1, e_2] = e_3; \\
e_1 \circ e_2 = -\alpha e_3, \quad e_2 \circ e_1 = \alpha e_3
\end{cases}
$\\

$\mathrm{G}_{07}$&$:$&$
\begin{cases}
[e_1, e_2] = e_3; \\
e_1 \circ e_1 = e_1, \quad e_1 \circ e_2 = -e_2, \quad e_2 \circ e_1 = e_2
\end{cases}
$\\

$\mathrm{G}_{08}^{\alpha,\beta}$&$:$&$
\begin{cases}
[e_1, e_2] = e_3; \\
e_1 \circ e_2 = \alpha e_3, \quad e_2 \circ e_1 = \beta e_3, \quad e_2 \circ e_2 = e_1
\end{cases}
$\\

$\mathrm{G}_{09}$&$:$&$
\begin{cases}
[e_1, e_2] = e_3; \\
e_2 \circ e_2 = e_2 + e_3, \quad e_3 \circ e_2 = e_3
\end{cases}
$\\

$\mathrm{G}_{10}$&$:$&$
\begin{cases}
[e_1, e_2] = e_3; \\
e_2 \circ e_2 = e_2, \quad e_3 \circ e_2 = e_3
\end{cases}
$\\

$\mathrm{G}_{11}^{\alpha,\beta}$&$:$&$
\begin{cases}
[e_1, e_2] = e_3; \\
e_1 \circ e_1 = \alpha e_3, \quad e_1 \circ e_2 = \beta e_3, \quad e_2 \circ e_2 = e_3
\end{cases}
$\\

$\mathrm{G}_{12}^{\alpha,\beta,\gamma}$&$:$&$
\begin{cases}
[e_1, e_3] = e_1, \quad [e_2, e_3] = e_2; \\
e_1 \circ e_3 = \alpha e_1, \quad e_2 \circ e_3 = \alpha e_2, \quad e_3 \circ e_1 = \beta e_2, \quad 
e_3 \circ e_2 = e_1 + \gamma e_2, \quad e_3 \circ e_3 = \alpha e_3
\end{cases}
$\\

$\mathrm{G}_{13}^{\alpha,\beta}$&$:$&$
\begin{cases}
[e_1, e_3] = e_1, \quad [e_2, e_3] = e_2; \\
e_1 \circ e_3 = \alpha e_1, \quad e_2 \circ e_3 = \alpha e_2, \quad 
e_3 \circ e_2 = e_1 + \beta e_2, \quad e_3 \circ e_3 = e_2 + \alpha e_3
\end{cases}
$\\

$\mathrm{G}_{14}^{\alpha,\beta}$&$:$&$
\begin{cases}
[e_1, e_3] = e_1, \quad [e_2, e_3] = e_2; \\
e_1 \circ e_3 = \alpha e_1, \quad e_2 \circ e_3 = \alpha e_2,\quad e_3 \circ e_1 = \beta e_1, \quad  e_3 \circ e_2 = \beta e_2, \quad e_3 \circ e_3 = \alpha e_3
\end{cases}
$\\

$\mathrm{G}_{15}^{\alpha}$&$:$&$
\begin{cases}
[e_1, e_3] = e_1, \quad [e_2, e_3] = e_2; \\
e_1 \circ e_3 = \alpha e_1, \quad e_2 \circ e_3 = \alpha e_2, \quad e_3 \circ e_3 = e_1 + \alpha e_3
\end{cases}
$\\

$\mathrm{G}_{16}^{\alpha,\beta,\gamma,\delta
}$&$:$&$
\begin{cases}
[e_1,e_3]=e_1,\quad [e_2,e_3]=\alpha e_2;\\
e_1\circ e_3=\beta e_1,\quad e_2\circ e_3=\beta e_2,\quad 
e_3\circ e_1=\gamma e_1,\quad e_3\circ e_2=\delta e_2,\quad e_3\circ e_3=\beta e_3
\end{cases}
$\\

$\mathrm{G}_{17}^{\alpha,\beta,\gamma
}$&$:$&$
\begin{cases}
[e_1,e_3]=e_1,\quad [e_2,e_3]=\alpha e_2; \\
e_1\circ e_3=\beta e_1,\quad e_2\circ e_3=\beta e_2,\quad  
e_3\circ e_1=\gamma e_1,\quad e_3\circ e_3=e_2+\beta e_3
\end{cases}
$\\

$\mathrm{G}_{18}^{\alpha,\beta
}$&$:$&$
\begin{cases}
[e_1,e_3]=e_1,\quad [e_2,e_3]=\alpha e_2; \\
e_1\circ e_3=\beta e_1,\quad e_2\circ e_3=\beta e_2,\quad e_3\circ e_3=e_1+e_2+\beta e_3
\end{cases}
$\\

$\mathrm{G}_{19}^{\alpha,\beta
}$&$:$&$
\begin{cases}
[e_1,e_3]=e_1,\quad [e_2,e_3]=\alpha e_2; \\
e_1\circ e_3=\beta e_1+e_2,\quad e_2\circ e_3=\beta e_2 ,\quad e_3\circ e_1=\frac{1}{\alpha-1}e_2,\\ e_3\circ e_2=-\beta e_2, \quad e_3\circ e_3=\beta e_3
\end{cases}
$\\

$\mathrm{G}_{20}^{\alpha,\beta
}$&$:$&$
\begin{cases}
[e_1,e_3]=e_1,\quad [e_2,e_3]=\alpha e_2; \\
e_1\circ e_3=\beta e_1+e_2,\quad e_2\circ e_3=\beta e_2 ,\quad e_3\circ e_1=\frac{1}{\alpha-1}e_2,\\ e_3\circ e_2=-\beta e_2 ,\quad e_3\circ e_3=e_1+\beta e_3
\end{cases}
$\\

$\mathrm{G}_{21}^{\alpha,\beta
}$&$:$&$
\begin{cases}
[e_1,e_3]=e_1,\quad [e_2,e_3]=2e_2;\\
e_1\circ e_1=e_2,\quad e_1\circ e_3=\alpha e_1,\quad e_2\circ e_3=\alpha e_2,\\  
e_3\circ e_1=\beta e_1,\quad
e_3\circ e_2=(-\alpha+2\beta)e_2,\quad e_3\circ e_3=\alpha e_3
\end{cases}
$\\

$\mathrm{G}_{22}^{\alpha}$&$:$&$
\begin{cases}
[e_1,e_3]=e_1,\quad [e_2,e_3]=\frac12e_2; \\
e_2\circ e_2=e_1,\quad e_1\circ e_3=\alpha e_1,\quad e_2\circ e_3=\alpha e_2 ,\quad  e_3\circ e_2=\frac{\alpha}{2}e_2,\quad e_3\circ e_3=\alpha e_3
\end{cases}
$\\

$\mathrm{G}_{23}^{\alpha}$&$:$&$
\begin{cases}
[e_1,e_3]=e_1,\quad [e_2,e_3]=\frac12e_2; \\
e_2\circ e_2=e_1,\quad e_1\circ e_3=\alpha e_1,\quad e_2\circ e_3=\alpha e_2 ,\quad
e_3\circ e_2=\frac{\alpha}{2}e_2,\quad e_3\circ e_3=e_1+\alpha e_3
\end{cases}
$\\

$\mathrm{G}_{24}^{\alpha,\beta,\gamma}$&$:$&$
\begin{cases}
[e_1,e_3]=e_1,\quad [e_2,e_3]=-e_2;\\
e_1\circ e_3=\alpha e_1,\quad e_2\circ e_3=\alpha e_2, \quad 
e_3\circ e_1=\beta e_1,\quad e_3\circ e_2=\gamma e_2,\quad e_3\circ e_3=\alpha e_3
\end{cases}
$\\

$\mathrm{G}_{25}^{\alpha,\beta}$&$:$&$
\begin{cases}
[e_1,e_3]=e_1,\quad [e_2,e_3]=-e_2; \\
e_1\circ e_3=\alpha e_1,\quad e_2\circ e_3=\alpha e_2 ,\quad e_3\circ e_2=\beta e_2,\quad e_3\circ e_3=e_1+\alpha e_3
\end{cases}
$\\

$\mathrm{G}_{26}^{\alpha}$&$:$&$
\begin{cases}
[e_1,e_3]=e_1,\quad [e_2,e_3]=-e_2;\\
e_1\circ e_3=\alpha e_1,\quad e_2\circ e_3=\alpha e_2,\quad e_3\circ e_3=e_1+e_2+\alpha e_3
\end{cases}
$\\

$\mathrm{G}_{27}^{\alpha}$&$:$&$
\begin{cases}
[e_1,e_3]=e_1,\quad [e_2,e_3]=-e_2; \\
e_1\circ e_3=\alpha e_1,\quad e_2\circ e_3=e_1+\alpha e_2 ,\quad
e_3\circ e_1=-\alpha e_1,\\ 
e_3\circ e_2=-\frac12e_1,\quad e_3\circ e_3=\alpha e_3
\end{cases}
$\\

$\mathrm{G}_{28}^{\alpha}$&$:$&$
\begin{cases}
[e_1,e_3]=e_1,\quad [e_2,e_3]=-e_2; \\
e_1\circ e_3=\alpha e_1,\quad e_2\circ e_3=e_1+\alpha e_2, \quad
e_3\circ e_1=-\alpha e_1,\\  
e_3\circ e_2=-\frac12e_1,\quad e_3\circ e_3=e_2+\alpha e_3
\end{cases}
$\\

$\mathrm{G}_{29}^{\alpha,\beta,\gamma}$&$:$&$
\begin{cases}
[e_1,e_3]=e_1+e_2,\quad [e_2,e_3]=e_2;\\
e_1\circ e_3=\alpha e_1,\quad e_2\circ e_3=\alpha e_2,\quad e_3\circ e_1=\beta e_1+\gamma e_2,\\
e_3\circ e_2=\beta e_2,\quad e_3\circ e_3=\alpha e_3
\end{cases}
$\\

$\mathrm{G}_{30}^{\alpha,\beta}$&$:$&$
\begin{cases}
[e_1,e_3]=e_1+e_2,\quad [e_2,e_3]=e_2; \\
e_1\circ e_3=\alpha e_1,\quad e_2\circ e_3=\alpha e_2 ,\quad e_3\circ e_1=\beta e_2,\quad e_3\circ e_3=e_1+\alpha e_3
\end{cases}
$\\

$\mathrm{G}_{31}^{\alpha}$&$:$&$
\begin{cases}
[e_1,e_3]=e_1+e_2,\quad [e_2,e_3]=e_2;\\
e_1\circ e_3=\alpha e_1,\quad e_2\circ e_3=\alpha e_2,\quad e_3\circ e_3=e_2+\alpha e_3
\end{cases}
$\\

$\mathrm{G}_{32}^{\alpha
}$&$:$&$
\begin{cases}
[e_1,e_3]=e_1+e_2,\quad [e_2,e_3]=e_2;\\
e_1\circ e_3=\alpha(e_1+e_2),\quad e_2\circ e_3=\alpha e_2,\quad e_3\circ e_2=-\alpha e_2,\quad e_3\circ e_3=\alpha e_3
\end{cases}
$\\

$\mathrm{G}_{33}^{\alpha
}$&$:$&$
\begin{cases}
[e_1,e_3]=e_1+e_2,\quad [e_2,e_3]=e_2;\\
e_1\circ e_3=\alpha(e_1+e_2),\quad e_2\circ e_3=\alpha e_2,\quad e_3\circ e_2=-\alpha e_2,\quad e_3\circ e_3=e_1+\alpha e_3
\end{cases}
$\\

$\mathrm{G}_{34}^{\alpha}$&$:$&$
\begin{cases}
[e_1,e_3]=e_1;\\
e_1\circ e_2=e_1,\quad e_2\circ e_1=\alpha e_1,\quad e_2\circ e_2=e_2,\quad e_3\circ e_2=e_3
\end{cases}
$\\

$\mathrm{G}_{35}$&$:$&$
\begin{cases}
[e_1,e_3]=e_1; \\
e_1\circ e_2=e_1,\quad e_2\circ e_2=e_2 ,\quad e_2\circ e_3=e_1,\quad e_3\circ e_2=e_3
\end{cases}
$\\

$\mathrm{G}_{36}$&$:$&$
\begin{cases}
[e_1,e_3]=e_1; \\
e_1\circ e_2=e_1,\quad e_2\circ e_1=-e_1,\quad e_2\circ e_2=e_2 ,\\ e_3\circ e_2=e_3,\quad e_3\circ e_3=e_1
\end{cases}
$\\

$\mathrm{G}_{37}^{\alpha}$&$:$&$
\begin{cases}
[e_1,e_3]=e_1;\\
e_1\circ e_3=e_2,\quad e_3\circ e_1=-e_2,\quad e_3\circ e_3=e_1+\alpha e_2
\end{cases}
$\\

$\mathrm{G}_{38}$&$:$&$
\begin{cases}
[e_1,e_3]=e_1;\\
e_1\circ e_3=e_2,\quad e_3\circ e_1=-e_2,\quad e_3\circ e_3=e_2
\end{cases}
$\\

$\mathrm{G}_{39}^{\alpha
}$&$:$&$
\begin{cases}
[e_1,e_3]=e_1;\\
e_1\circ e_3=\alpha e_1+e_2,\quad e_3\circ e_1=-\alpha e_1-e_2,\quad e_3\circ e_3=\alpha e_3
\end{cases}
$\\

$\mathrm{G}_{40}^{\alpha}$&$:$&$
\begin{cases}
[e_1,e_3]=e_1; \\
e_1\circ e_3=\alpha e_1+e_2,\quad e_2\circ e_3=\alpha e_2,\quad 
e_3\circ e_1=-e_2,\\   
e_3\circ e_2=-\alpha e_2,\quad e_3\circ e_3=\alpha e_3
\end{cases}
$\\

$\mathrm{G}_{41}^{\alpha
}$&$:$&$
\begin{cases}
[e_1,e_3]=e_1; \\
e_1\circ e_3=\alpha e_1+e_2,\quad e_2\circ e_3=\alpha e_2, \quad 
e_3\circ e_1=-e_2,\\  
e_3\circ e_2=-\alpha e_2,\quad e_3\circ e_3=e_1+\alpha e_3
\end{cases}
$\\

$\mathrm{G}_{42}$&$:$&$
\begin{cases}
[e_1,e_3]=e_1;\\
e_2\circ e_1=e_1
\end{cases}
$\\

$\mathrm{G}_{43}^{\alpha,\beta}$&$:$&$
\begin{cases}
[e_1,e_3]=e_1;\\
e_1\circ e_3=\alpha e_1,\quad e_2\circ e_2=e_2,\quad e_3\circ e_1=\beta e_1,\quad e_3\circ e_3=\alpha e_3
\end{cases}
$\\

$\mathrm{G}_{44}^{\alpha}$&$:$&$
\begin{cases}
[e_1,e_3]=e_1; \\
e_1\circ e_3=\alpha e_1,\quad e_2\circ e_2=e_2,\quad e_3\circ e_3=e_1+\alpha e_3
\end{cases}
$\\

$\mathrm{G}_{45}^{\alpha,\beta,\gamma}$&$:$&$
\begin{cases}
[e_1,e_3]=e_1;\\
e_1\circ e_3=\alpha e_1,\quad e_2\circ e_3=\alpha e_2,\quad e_3\circ e_1=\beta e_1,\quad e_3\circ e_2=\gamma e_2,\quad e_3\circ e_3=\alpha e_3
\end{cases}
$\\

$\mathrm{G}_{46}^{\alpha,\beta}$&$:$&$
\begin{cases}
[e_1,e_3]=e_1 \\
e_1\circ e_3=\alpha e_1,\quad e_2\circ e_3=\alpha e_2 ,\quad e_3\circ e_1=\beta e_1,\quad e_3\circ e_3=e_2+\alpha e_3
\end{cases}
$\\

$\mathrm{G}_{47}^{\alpha,\beta}$&$:$&$
\begin{cases}
[e_1,e_3]=e_1; \\
e_1\circ e_3=\alpha e_1,\quad e_2\circ e_3=\alpha e_2 ,\quad e_3\circ e_2=\beta e_2,\quad e_3\circ e_3=e_1+\alpha e_3
\end{cases}
$\\

$\mathrm{G}_{48}^{\alpha}$&$:$&$
\begin{cases}
[e_1,e_3]=e_1 ;\\
e_1\circ e_3=\alpha e_1,\quad e_2\circ e_3=\alpha e_2,\quad e_3\circ e_3=e_1+e_2+\alpha e_3
\end{cases}
$\\

$\mathrm{G}_{49}^{\alpha,\beta
}$&$:$&$
\begin{cases}
[e_1,e_3]=e_1;\\
e_1\circ e_3=\alpha e_1,\quad e_3\circ e_1=\beta e_1,\quad e_3\circ e_3=\alpha e_3
\end{cases}
$\\

$\mathrm{G}_{50}^{\alpha
}$&$:$&$
\begin{cases}
[e_1,e_3]=e_1; \\
e_1\circ e_3=\alpha e_1,\quad e_3\circ e_3=e_1+\alpha e_3
\end{cases}
$\\

$\mathrm{G}_{51}^{\alpha}$&$:$&$
\begin{cases}
[e_1,e_3]=e_1; \\
e_1\circ e_3=\alpha e_1,\quad e_2\circ e_3=e_1,\quad e_3\circ e_3=\alpha e_3
\end{cases}
$\\

$\mathrm{G}_{52}^{\alpha}$&$:$&$
\begin{cases}
[e_1,e_3]=e_1; \\
e_1\circ e_3=\alpha e_1,\quad e_2\circ e_3=e_1,\quad e_3\circ e_3=e_2+\alpha e_3
\end{cases}
$\\

$\mathrm{G}_{53}^{\alpha}$&$:$&$
\begin{cases}
[e_1,e_3]=e_1; \\
e_1\circ e_3=\alpha e_1,\quad e_2\circ e_3=e_1+\alpha e_2 ,\quad e_3\circ e_1=-\alpha e_1,\quad e_3\circ e_3=\alpha e_3
\end{cases}
$\\

$\mathrm{G}_{54}^{\alpha}$&$:$&$
\begin{cases}
[e_1,e_3]=e_1; \\
e_1\circ e_3=\alpha e_1,\quad e_2\circ e_3=e_1+\alpha e_2 ,\quad e_3\circ e_1=-\alpha e_1,\quad e_3\circ e_3=e_2+\alpha e_3
\end{cases}
$\\

$\mathrm{G}_{55}$&$:$&$
\begin{cases}
[e_1, e_2] = e_3, \quad [e_1, e_3] = -2e_1, \quad [e_2, e_3] = 2e_2
\end{cases}
$\\

$\mathrm{G}_{56}$&$:$&$
\begin{cases}
[e_1, e_2] = e_3, \quad [e_1, e_3] = -2e_1, \quad [e_2, e_3] = 2e_2; \\[1ex]
e_2 \circ e_2 = -\frac{1}{4} e_3, \quad e_2 \circ e_3 = e_1, \quad e_3 \circ e_2 = -\frac{1}{2} e_1
\end{cases}
$\\

\end{longtable}
\endgroup

\end{theorem}

\subsection{\normalsize The geometric classification of $3$-dimensional Gelfand--Dorfman algebras}

\begin{theoremB}

The variety of \(3\)-dimensional Gelfand--Dorfman algebras has dimension \(11\) and it has \(10\) irreducible components defined by

\begin{longtable}{l}
$\mathcal{C}_{1}=\overline{\mathcal{O}\bigl(\mathrm{G}_{16}^{\alpha,\beta,\gamma,\delta}\bigr)},\
\mathcal{C}_{2}=\overline{\mathcal{O}\bigl(\mathrm{G}_{20}^{\alpha,\beta}\bigr)},\
\mathcal{C}_{3}=\overline{\mathcal{O}\bigl(\mathrm{G}_{21}^{\alpha,\beta}\bigr)},\
\mathcal{C}_{4}=\overline{\mathcal{O}\bigl(\mathrm{G}_{34}^{\alpha}\bigr)},\
\mathcal{C}_{5}=\overline{\mathcal{O}\bigl(\mathrm{G}_{36}\bigr)},$\\[2pt]
$\mathcal{C}_{6}=\overline{\mathcal{O}\bigl(\mathrm{G}_{39}^{\alpha}\bigr)},\
\mathcal{C}_{7}=\overline{\mathcal{O}\bigl(\mathrm{G}_{43}^{\alpha,\beta}\bigr)},\
\mathcal{C}_{8}=\overline{\mathcal{O}\bigl(\mathrm{G}_{52}^{\alpha}\bigr)},\
\mathcal{C}_{9}=\overline{\mathcal{O}\bigl(\mathrm{G}_{56}\bigr)},\quad\text{and}\quad
\mathcal{C}_{10}=\overline{\mathcal{O}\bigl(\mathrm{N}_{01}\bigr)},$
\end{longtable}

\noindent where $\mathrm{N}_{01}$ has the following multiplication table 
\begin{center}
$e_1\circ e_1=e_1,\ e_2\circ e_2=e_2,\ e_3\circ e_3=e_3.$
\end{center}\noindent
In particular, there are \(3\) rigid algebras: $\mathrm{G}_{36}$, $\mathrm{G}_{56}$ and $\mathrm{N}_{01}$.
    
\end{theoremB}

\begin{proof}
Thanks to Proposition \ref{propN}, each GD algebra with trivial anticommutative multiplication is in the orbit closure of one of the following families of algebras  
${\rm N}_{01},$ 
${\rm N}_{07}^{\alpha},$ 
${\rm N}_{10}^{\alpha},$ or 
${\rm N}_{11}^{\alpha,\beta}.$
The dimensions of the irreducible components are as follows
\begin{longtable}{rll}
$\dim  \mathcal{C}_{1}$&$=$&$11,$ \\ 
$\dim  \mathcal{C}_{2}=\dim  \mathcal{C}_{3}=\dim  \mathcal{C}_{7}$&$=$&$10,$\\  
$\dim  \mathcal{C}_{4}=\dim  \mathcal{C}_{5}=\dim  \mathcal{C}_{6}=\dim  \mathcal{C}_{8}=\dim  \mathcal{C}_{10}$&$=$&$9,$\\
$\dim  \mathcal{C}_{9}$&$=$&$8.$\\

\end{longtable}

The following list of degenerations shows that every $3$-dimensional Gelfand--Dorfman algebra belongs to $\mathcal{C}_1\cup\ldots\cup\mathcal{C}_{10}$.

\begin{longtable}{lcl}
$\mathrm{G}_{43}^{\frac{\alpha}{t},\ \frac{\alpha+1}{t}}$ & $\xrightarrow{\big(te_3,\ e_1,\ e_2\big)}$ & ${\rm N}_{07}^{\alpha}$\\[0.5ex]\hline
$\mathrm{G}_{21}^{\frac{\alpha}{t},\ \frac{2\alpha+1}{2t}}$ & $\xrightarrow{\big(te_3,\ e_2,\ e_1\big)}$ & ${\rm N}_{10}^{\alpha}$\\[0.5ex]\hline

$\mathrm{G}_{16}^{2,\ \frac{\beta}{t},\ \frac{\alpha+\beta}{t},\ \frac{\beta+1}{t}}$ & $\xrightarrow{\big(te_3,\ e_2,\ e_1\big)}$ & ${\rm N}_{11}^{\alpha,\beta}$\\[0.5ex]\hline

$\mathrm{G}_{16}^{t, 0, \frac{1}{t}, \alpha+\frac{1}{t}}$ 
  & $\xrightarrow{\big(te_3,\ -\frac{1}{t}e_1+e_2,\ e_1\big)}$ 
  & $\mathrm{G}_{01}^{\alpha}$ \\[0.5ex]
\hline
$\mathrm{G}_{43}^{-\frac{1}{t}, \alpha-\frac{1}{t}}$ 
  & $\xrightarrow{\big(-\frac{1}{t}e_1+t e_2,\ e_2-t e_3,\ e_1\big)}$ 
  & $\mathrm{G}_{02}^{\alpha}$ \\[0.5ex]
\hline
$\mathrm{G}_{16}^{t, -\frac{1}{t}, -\frac{\alpha}{t}, \beta-\frac{\alpha}{t}}$ 
  & $\xrightarrow{\big(-t e_3,\ \frac{1}{t}e_1+e_2,\ e_1\big)}$ 
  & $\mathrm{G}_{03}^{\alpha,\beta}$ \\[0.5ex]
\hline
$\mathrm{G}_{16}^{t, \frac{1}{t}, t, \alpha+t}$ 
  & $\xrightarrow{\big(-\frac{1}{t^{3}}e_1+\frac{1}{t(\alpha+t)}e_2+t e_3,\ -\frac{1}{t}e_1+e_2,\ e_1\big)}$ 
  & $\mathrm{G}_{04}^{\alpha}$ \\[0.5ex]
\hline
$\mathrm{G}_{34}^{t}$ 
  & $\xrightarrow{\big(e_1+e_2,\ t e_3,\ t e_1\big)}$ 
  & $\mathrm{G}_{05}$ \\[0.5ex]
\hline
$\mathrm{G}_{43}^{t-\alpha, \alpha-t}$ 
  & $\xrightarrow{\big(t e_3,\ e_1+t^2 e_2,\ -t e_1\big)}$ 
  & $\mathrm{G}_{06}^{\alpha}$ \\[0.5ex]
\hline
$\mathrm{G}_{39}^{-\frac{1}{t}}$ 
  & $\xrightarrow{\big(-t e_3,\ \frac{1}{t}e_1-e_2,\ e_1\big)}$ 
  & $\mathrm{G}_{07}$ \\[0.5ex]
\hline
$\mathrm{G}_{20}^{\frac{1}{t}, -\frac{\beta}{t}}$ 
  & $\xrightarrow{\big(t^4 e_1+\frac{t^5}{\alpha+t}e_2,\ -\frac{t^3(1-t)}{\alpha+t}e_1-t^2 e_3,\ -\frac{t^6}{\alpha+t}e_2\big)}$ 
  & $\mathrm{G}_{08}^{\alpha,\beta}$ \\[0.5ex]
\hline
$\mathrm{G}_{52}^{\frac{1}{t}}$ 
  & $\xrightarrow{\big(t^2 e_1 - t e_2,\ t^2 e_2 + t e_3,\ t^3 e_1\big)}$ 
  & $\mathrm{G}_{09}$ \\[0.5ex]
\hline
$\mathrm{G}_{52}^{\frac{1}{t}}$ 
  & $\xrightarrow{\big(t e_1 - e_2,\ t^2 e_2 + t e_3,\ t^2 e_1\big)}$ 
  & $\mathrm{G}_{10}$ \\[0.5ex]
\hline

    $\mathrm{G}_{43}^{\frac{\beta+\sqrt{\beta^{2}-4\alpha}}{2},\ \frac{-\beta+\sqrt{\beta^{2}-4\alpha}}{2}}$
  & $\xrightarrow{\quad\quad\quad\quad\quad\quad\quad}$
  & $\mathrm{G}_{11}^{\alpha,\beta}$ \\
  
\multicolumn{3}{c}{$E_1^t= \frac{\beta+\sqrt{\beta^{2}-4\alpha}}{2}(e_1+t e_2)+t e_3$
\qquad $E_2^t= e_1+t e_2$ \qquad
$E_3^t= -t e_1$} \\[0.5ex]
\hline

\multicolumn{3}{l}{$\mathrm{G}_{16}^{\frac{2+t\big(\gamma+\sqrt{\gamma^2+4\beta}\big)}{2+t\big(\gamma-\sqrt{\gamma^2+4\beta}\big)},\ \frac{2\alpha}{2+t\big(\gamma-\sqrt{\gamma^2+4\beta}\big)},\ \frac{\gamma-\sqrt{\gamma^2+4\beta}}{2+t\big(\gamma-\sqrt{\gamma^2+4\beta}\big)},\ \frac{\gamma+\sqrt{\gamma^2+4\beta}}{2+t\big(\gamma-\sqrt{\gamma^2+4\beta}\big)}} \xrightarrow{\qquad\quad\ } \mathrm{G}_{12}^{\alpha,\beta,\gamma}$} \\

 \multicolumn{3}{c}{\small $E_1^t= \frac{\big(\gamma+\sqrt{\gamma^2+4\beta}\big)e_1
-\big(\gamma-\sqrt{\gamma^2+4\beta}\big)e_2}{2\sqrt{\gamma^2+4\beta}}$\qquad  $E_2^t= \frac{-e_1+e_2}{\sqrt{\gamma^2+4\beta}}$ \qquad
$E_3^t= \big(1+t\frac{\gamma-\sqrt{\gamma^2+4\beta}}{2}\big)e_3$} \\[0.5ex]
\hline

$\mathrm{G}_{16}^{\frac{1+\beta t+2t^2}{1+t^2},\ \frac{\alpha}{1+t^2},\ \frac{t}{1+t^2},\ \frac{\beta+2t}{1+t^2}}$ 
  & $\xrightarrow{\big(e_1,\ \frac{-e_1+e_2}{\beta+t},\ -\frac{e_1}{t(\beta+t)}+\frac{e_2}{(\beta+2t)(\beta+t)}+(1+t^2)e_3\big)}$ 
  & $\mathrm{G}_{13}^{\alpha,\beta}$ \\[0.5ex]
\hline
$\mathrm{G}_{16}^{1+t, \alpha, \beta+t, \beta+t}$ 
  & $\xrightarrow{\big(e_1,\ e_2,\ e_3\big)}$ 
  & $\mathrm{G}_{14}^{\alpha,\beta}$ \\[0.5ex]
\hline
$\mathrm{G}_{16}^{1+t, \alpha, t, t}$ 
  & $\xrightarrow{\big(e_1,\ e_2,\ t^{-1}e_1+e_3\big)}$ 
  & $\mathrm{G}_{15}^{\alpha}$ \\[0.5ex]
\hline
$\mathrm{G}_{16}^{\alpha, \beta, \gamma, t}$ 
  & $\xrightarrow{\big(e_1,\ e_2,\ e_3+t^{-1}e_2\big)}$ 
  & $\mathrm{G}_{17}^{\alpha,\beta,\gamma}$ \\[0.5ex]
\hline
$\mathrm{G}_{16}^{\alpha, \beta, t, t}$ 
  & $\xrightarrow{\big(e_1,\ e_2,\ e_3+t^{-1}e_1+t^{-1}e_2\big)}$ 
  & $\mathrm{G}_{18}^{\alpha,\beta}$ \\[0.5ex]
\hline
$\mathrm{G}_{20}^{\alpha, \beta}$ 
  & $\xrightarrow{\big(t^{-1}e_1,\ t^{-1}e_2,\ e_3\big)}$ 
  & $\mathrm{G}_{19}^{\alpha,\beta}$ \\[0.5ex]
\hline
$\mathrm{G}_{21}^{2\alpha+t, \alpha}$ 
  & $\xrightarrow{\big(e_2,\ e_1,\ \frac{1}{2}e_3\big)}$ 
  & $\mathrm{G}_{22}^{\alpha}$ \\[0.5ex]
\hline
$\mathrm{G}_{21}^{2\alpha+t, \alpha}$ 
  & $\xrightarrow{\big(\frac{e_2}{4},\ \frac{e_1}{2},\ \frac{1}{2}e_3-\frac{1}{2t}e_2\big)}$ 
  & $\mathrm{G}_{23}^{\alpha}$ \\[0.5ex]
\hline
$\mathrm{G}_{16}^{-1+t, \alpha, \beta, \gamma}$ 
  & $\xrightarrow{\big(e_1,\ e_2,\ e_3\big)}$ 
  & $\mathrm{G}_{24}^{\alpha,\beta,\gamma}$ \\[0.5ex]
\hline
$\mathrm{G}_{16}^{-1+t, \alpha, t, \beta}$ 
  & $\xrightarrow{\big(e_1,\ e_2,\ e_3+t^{-1}e_1\big)}$ 
  & $\mathrm{G}_{25}^{\alpha,\beta}$ \\[0.5ex]
\hline
$\mathrm{G}_{16}^{-1+t, \alpha, t, t}$ 
  & $\xrightarrow{\big(e_1,\ e_2,\ e_3+t^{-1}(e_1+e_2)\big)}$ 
  & $\mathrm{G}_{26}^{\alpha}$ \\[0.5ex]
\hline
$\mathrm{G}_{20}^{\frac{1}{t-1}, \frac{\alpha}{t-1}}$ 
  & $\xrightarrow{\big(\frac{(t-1)^3}{t} e_2,\ \frac{(t-1)^2}{t} e_1,\ (t-1) e_3\big)}$ 
  & $\mathrm{G}_{27}^{\alpha}$ \\[0.5ex]
\hline
$\mathrm{G}_{20}^{\frac{1}{t-1}, \frac{\alpha}{t-1}}$ 
  & $\xrightarrow{\big((t-1)^3 e_2,\ (t-1)^2 e_1,\ (t-1) e_3\big)}$ 
  & $\mathrm{G}_{28}^{\alpha}$ \\[0.5ex]
\hline
$\mathrm{G}_{16}^{1+t, \alpha, \beta, \beta+\gamma t}$ 
  & $\xrightarrow{\big(-t^{-1}e_1+t^{-1}e_2,\ e_2,\ e_3\big)}$ 
  & $\mathrm{G}_{29}^{\alpha,\beta,\gamma}$ \\[0.5ex]
\hline
$\mathrm{G}_{16}^{1+t, \alpha, t, t(1+\beta+t)}$ 
  & $\xrightarrow{\big(-t^{-1}e_1+t^{-1}e_2,\ e_2,\ -t^{-2}e_1+\frac{1}{t^2(1+\beta+t)}e_2+e_3\big)}$ 
  & $\mathrm{G}_{30}^{\alpha,\beta}$ \\[0.5ex]
\hline
$\mathrm{G}_{16}^{1+t, \alpha, t, t}$ 
  & $\xrightarrow{\big(-t^{-1}e_1+t^{-1}e_2,\ e_2,\ t^{-1}e_2+e_3\big)}$ 
  & $\mathrm{G}_{31}^{\alpha}$ \\[0.5ex]
\hline
$\mathrm{G}_{20}^{1+t,\alpha}$ 
  & $\xrightarrow{\big(t^{-2}e_1+(\alpha t^3)^{-1}e_2,\ (\alpha t^2)^{-1}e_2,\ e_3\big)}$ 
  & $\mathrm{G}_{32}^{\alpha}$ \\[0.5ex]
\hline
$\mathrm{G}_{20}^{1+t, \alpha}$ 
  & $\xrightarrow{\big(e_1+(\alpha t)^{-1}e_2,\ \alpha^{-1}e_2,\ (\alpha(1+t))^{-1}e_1+e_3\big)}$ 
  & $\mathrm{G}_{33}^{\alpha}$ \\[0.5ex]
\hline
$\mathrm{G}_{34}^{t}$ 
  & $\xrightarrow{\big(e_1,\ e_2,\ e_3+t^{-1}e_1\big)}$ 
  & $\mathrm{G}_{35}$ \\[0.5ex]
\hline
$\mathrm{G}_{39}^{t}$ 
  & $\xrightarrow{\big(e_1,\ e_2,\ e_3-t^{-1}e_1-\alpha t^{-1}e_2\big)}$ 
  & $\mathrm{G}_{37}^{\alpha}$ \\[0.5ex]
\hline
$\mathrm{G}_{39}^{t}$ 
  & $\xrightarrow{\big(e_1,\ e_2,\ e_3-t^{-1}e_2\big)}$ 
  & $\mathrm{G}_{38}$ \\[0.5ex]
\hline
$\mathrm{G}_{20}^{t, \alpha}$ 
  & $\xrightarrow{\big(\frac{e_1}{t},\ (t(1-t))^{-1}e_2,\ e_3\big)}$ 
  & $\mathrm{G}_{40}^{\alpha}$ \\[0.5ex]
\hline
$\mathrm{G}_{20}^{t, \alpha}$ 
  & $\xrightarrow{\big((1+t)^{-1}e_1,\ ((1+t)(1-t))^{-1}e_2,\ e_3\big)}$ 
  & $\mathrm{G}_{41}^{\alpha}$ \\[0.5ex]
\hline
$\mathrm{G}_{34}^{t^{-1}}$ 
  & $\xrightarrow{\big(e_1,\ t e_2,\ e_3\big)}$ 
  & $\mathrm{G}_{42}$ \\[0.5ex]
\hline
$\mathrm{G}_{43}^{\alpha, t}$ 
  & $\xrightarrow{\big(e_1,\ e_2,\ e_3+t^{-1}e_1\big)}$ 
  & $\mathrm{G}_{44}^{\alpha}$ \\[0.5ex]
\hline
$\mathrm{G}_{16}^{t, \alpha, \beta, \gamma}$ 
  & $\xrightarrow{\big(e_1,\ e_2,\ e_3\big)}$ 
  & $\mathrm{G}_{45}^{\alpha,\beta,\gamma}$ \\[0.5ex]
\hline
$\mathrm{G}_{16}^{t, \alpha, \beta, t}$ 
  & $\xrightarrow{\big(e_1,\ e_2,\ e_3+t^{-1}e_2\big)}$ 
  & $\mathrm{G}_{46}^{\alpha,\beta}$ \\[0.5ex]
\hline

$\mathrm{G}_{16}^{t, \alpha, t, \beta}$ 
  & $\xrightarrow{\big(e_1,\ e_2,\ e_3+t^{-1}e_1\big)}$ 
  & $\mathrm{G}_{47}^{\alpha,\beta}$ \\[0.5ex]
\hline
$\mathrm{G}_{16}^{t, \alpha, t, t}$ 
  & $\xrightarrow{\big(e_1,\ e_2,\ e_3+t^{-1}e_1+t^{-1}e_2\big)}$ 
  & $\mathrm{G}_{48}^{\alpha}$ \\[0.5ex]
\hline
$\mathrm{G}_{43}^{\alpha, \beta}$ 
  & $\xrightarrow{\big(e_1,\ t e_2,\ e_3\big)}$ 
  & $\mathrm{G}_{49}^{\alpha,\beta}$ \\[0.5ex]
\hline
$\mathrm{G}_{43}^{\alpha, t}$ 
  & $\xrightarrow{\big(e_1,\ t e_2,\ e_3+t^{-1}e_1\big)}$ 
  & $\mathrm{G}_{50}^{\alpha}$ \\[0.5ex]
\hline
$\mathrm{G}_{52}^{\alpha}$ 
  & $\xrightarrow{\big(t^{-1}e_1,\ t^{-1}e_2,\ e_3\big)}$ 
  & $\mathrm{G}_{51}^{\alpha}$ \\[0.5ex]
\hline
$\mathrm{G}_{20}^{t^{-1},\alpha t^{-1}}$ 
  & $\xrightarrow{\big(t^2e_2,\ te_1,\ t e_3\big)}$ 
  & $\mathrm{G}_{53}^{\alpha}$ \\[0.5ex]
\hline
$\mathrm{G}_{20}^{t^{-1},\alpha t^{-1}}$ 
  & $\xrightarrow{\big(\frac{t^3}{1+t}e_2,\ \frac{t^2}{1+t}e_1,\ t e_3\big)}$ 
  & $\mathrm{G}_{54}^{\alpha}$ \\[0.5ex]
\hline
$\mathrm{G}_{56}$ 
  & $\xrightarrow{\big(t^{-1}e_1,\ t e_2,\ e_3\big)}$ 
  & $\mathrm{G}_{55}$ \\[0.5ex]
\hline

\end{longtable}
 {Let} us now prove the necessary non-degenerations. The closures $\mathcal{C}_1,\ldots,\mathcal{C}_{10}$ are irreducible, and no algebra of one of the families defining them is isomorphic to an algebra of another one. Hence, $\mathcal{C}_i\not\subseteq\mathcal{C}_j$ whenever $i\neq j$ and $\dim\mathcal{C}_i\geq\dim\mathcal{C}_j$, and it is sufficient to prove that none of the algebras (families) defining $\mathcal{C}_1,\ldots,\mathcal{C}_{10}$ belongs to the closure of a family of bigger dimension. This is done in the following table, where each non-degeneration is proved by Lemma \ref{main2} with the set ${\rm C}$ given in the second column. Here, $c_{ij}^k$ and $d_{ij}^k$ are the structure constants of $\circ$ and $[\cdot,\cdot]$, respectively; the indices $i,j,k,q$ take all values in $\{1,2,3\}$, and each family is denoted by its name with the parameters of the classification. The non-degenerations whose target is a family are proved for generic values of its parameters; this is sufficient to show that $\mathcal{C}_i\not\subseteq\mathcal{C}_j$.

\begin{longtable}{|rcl|l|}
\hline
\multicolumn{3}{|c|}{\textbf{Non-degeneration}} & \textbf{Relations defining} ${\rm C}$ \\
\hline
$\mathrm{G}_{16}^{\alpha,\beta,\gamma,\delta},\,\mathrm{G}_{20}^{\alpha,\beta},$&&&\\
$\mathrm{G}_{21}^{\alpha,\beta},\,\mathrm{G}_{43}^{\alpha,\beta}$&$\not\to$&${\rm N}_{01}$
& $c_{33}^1=c_{33}^2=c_{33}^3=0$ \\

\hline
$\mathrm{G}_{16}^{\alpha,\beta,\gamma,\delta},\,\mathrm{G}_{20}^{\alpha,\beta},\,\mathrm{G}_{21}^{\alpha,\beta},$ &&&\\
$\mathrm{G}_{34}^{\alpha},\,\mathrm{G}_{36},\,\mathrm{G}_{39}^{\alpha},$&&&\\
$\mathrm{G}_{43}^{\alpha,\beta},\,\mathrm{G}_{52}^{\alpha},\,{\rm N}_{01}$&$\not\to$&$\mathrm{G}_{56}$ 
& $d_{ij}^1=0$ \\

\hline
$\mathrm{G}_{16}^{\alpha,\beta,\gamma,\delta}$&$\not\to$&$\mathrm{G}_{20}^{\alpha,\beta},\,\mathrm{G}_{21}^{\alpha,\beta},\,\mathrm{G}_{34}^{\alpha},\,\mathrm{G}_{36},$ & 
$\sum\limits_{p}\big(c_{ki}^pd_{jp}^q-c_{ji}^pd_{kp}^q+2d_{kj}^pc_{pi}^q\big)=0,$\\ 

&&$\mathrm{G}_{39}^{\alpha},\,\mathrm{G}_{43}^{\alpha,\beta},\,\mathrm{G}_{52}^{\alpha}$&
$c_{22}^k=c_{23}^k=c_{32}^k=c_{33}^k=0$\\
\hline

$\mathrm{G}_{21}^{\alpha,\beta},\,\mathrm{G}_{43}^{\alpha,\beta}$&$\not\to$&
$\mathrm{G}_{34}^{\alpha},\,\mathrm{G}_{36},\,\mathrm{G}_{39}^{\alpha},\,\mathrm{G}_{52}^{\alpha}$ 
& $\sum\limits_{p}\big(c_{ki}^pd_{jp}^q-c_{ji}^pd_{kp}^q+2d_{kj}^pc_{pi}^q\big)=0$ \\
\hline

$\mathrm{G}_{20}^{\alpha,\beta}$&$\not\to$&$\mathrm{G}_{34}^{\alpha},\,\mathrm{G}_{36},\,\mathrm{G}_{39}^{\alpha},\,\mathrm{G}_{52}^{\alpha}$ 
&  $\sum\limits_{p}c_{ip}^p=0,\quad c_{i2}^1=c_{i2}^2=c_{i3}^1=c_{i3}^2=0,$\\ &&&$d_{13}^1=d_{13}^2=d_{23}^1=d_{23}^2=d_{23}^3=0$  \\
\hline
\end{longtable}
 
\noindent Each set ${\rm C}$ consists of all structures satisfying the listed relations; it is stable under the action of the invertible lower triangular matrices. The relations containing sums over $p$ are even ${\rm GL}(\mathbb{V})$-invariant: the relation in the third and fourth rows means that the identity $[y,z\circ x]-[z,y\circ x]+2[z,y]\circ x=0$ holds (take $x=e_i$, $y=e_j$, $z=e_k$), and the relation $\sum_p c_{ip}^p=0$ in the last row means that ${\rm tr}\,L_x=0$ for all $x$, where $L_x$ is the operator of left multiplication by $x$ with respect to $\circ$.

\medskip

\noindent Summarizing, $\mathcal{C}_1, \ldots, \mathcal{C}_{10}$ are the irreducible components of the variety of $3$-dimensional Gelfand--Dorfman algebras, and the largest of them, $\mathcal{C}_1$, has dimension $11$.

\end{proof}


\end{document}